\documentclass{article}

\usepackage{graphicx} 
\usepackage[utf8]{inputenc} 
\usepackage[T1]{fontenc}    
\usepackage[english]{babel} 
\usepackage{amsmath, amssymb, amsthm} 
\usepackage{geometry}       
\usepackage[colorlinks=true, allcolors=blue]{hyperref}    
\usepackage{tocloft}        
\usepackage{braket}
\usepackage{cleveref}
\usepackage{tikz}
\usepackage{nicematrix}
\usepackage{bbm}
\usepackage{enumitem}

\usepackage[round]{natbib}
\usepackage[textsize=small]{todonotes}

\newcommand{\R}{\mathbb{R}}
\newcommand{\N}{\mathbb{N}}
\newcommand{\1}{\mathbbm{1}}

\usepackage{authblk}

\newtheorem{theorem}{Theorem}[section]
\newtheorem{lemma}[theorem]{Lemma}

\newtheorem{definition}[theorem]{Definition}
\newtheorem{proposition}[theorem]{Proposition}

\newtheorem{remark}[theorem]{Remark}

\numberwithin{equation}{section}
\title{From Hyper Roughness to Jumps as $H \to - 1 \, / \, 2$}
\author[1]{Eduardo Abi Jaber\footnote{EAJ is grateful for the financial support from the Chaires FiME-FDD, Financial Risks, Deep Finance \& Statistics and Machine Learning and systematic methods in finance at Ecole Polytechnique.
}}

\author[1]{Elie Attal} 

\author{Mathieu Rosenbaum}

\affil[1]{Ecole Polytechnique, CMAP}

\begin{document}

\maketitle

\begin{abstract}
 We investigate the weak limit of the hyper-rough square-root process  as the Hurst index $H$ goes to $-1/2\,$. This limit corresponds to the fractional kernel $t^{H - 1 / 2}$ losing integrability at zero. We establish the joint convergence of the couple $(X, M)\,$, where $X$ is the hyper-rough process and $M$ the associated martingale, to a fully correlated inverse Gaussian Lévy jump process. This unveils the existence of a continuum between hyper-rough continuous models and jump processes, as a function of the Hurst index. Since we prove a convergence of continuous to discontinuous processes, the usual Skorokhod $J_1$ topology is not suitable for our problem. Instead,  we obtain the weak convergence in the Skorokhod $M_1$ topology for $X$ and in the non-Skorokhod $S$ topology for $M$.
\end{abstract}


\medskip
\noindent\textbf{MSC2020 subject classifications:} 60B10, 60F05 (primary), 60G22, 60G51 (secondary).

\section{Introduction}

Since the pioneering work of \cite*{mandelbrot1968fractional} on fractals and fractional Brownian motions, stochastic processes with rough sample paths have garnered significant attention due to their ability to capture irregularities observed in various natural phenomena.\\

The Hurst index $H \in (0\,,1/2]\,$, named after Harold Edwin Hurst, is the key parameter that quantifies the roughness or more precisely the low Hölder regularity of the sample paths of the process. One prominent example is the rough square-root process, which appears in diverse contexts, including  scaling
limits of branching processes in population genetics  leading to catalytic superprocesses  \citep*{dawson1994super, 
 mytnik2015uniqueness} and self–exciting Hawkes processes
in mathematical finance \citep*{jaisson2016rough, el2019characteristic}. The rough square-root  process is governed by the stochastic Volterra equation
\begin{equation}
\label{eq:rough_heston}
    V_t = V_0 + \int_0^t \, (t-s)^{H - 1/2}\, \left(\lambda \, \left(\theta - V_s\right)\, ds + \nu \, \sqrt{V_s}\, dW_s \right) \,,
\end{equation}
where $W$ is a standard Brownian motion, $V_0\,,\lambda\,,\theta \geq 0$ and $\nu \in \mathbb R\,$.  The equation \eqref{eq:rough_heston} admits a unique weak nonnegative solution $V$
 whenever $ H \in (0\,,1/2]\,$ (see, e.g., \citet*[Theorem~6.1]{abi2019affine}), that is, when the fractional kernel $K_H(t)= t^{H-1/2}\,$, is locally square-integrable so that the stochastic integral is well defined as an Itô integral. Under \eqref{eq:rough_heston}, 
$V$ exhibits sample paths with Hölder continuity of any order strictly less than 
$H$, making it less regular than a standard Brownian motion. For $H<1/2\,$, this justifies the use of the term \textit{rough} and in this case,  $V$ is neither Markovian nor a semimartingale.\\

Compared to the initial rough volatility processes introduced in mathematical finance \citep*{gatheral2018volatility}, the model \eqref{eq:rough_heston} has the advantage of allowing efficient pricing \citep{el2019characteristic}, hedging \citep{euch2018perfect}, as well as having natural microstructural foundations \citep{jusselin2020no}. This is related to a key feature of the rough square-root process, namely its affine Volterra structure, as described in \cite*{abi2019affine, abi2021weak}. This property allows for extensions beyond the rough regime to encompass nonpositive Hurst indices. In particular, by considering the integrated process $X=\int_0^{\cdot}V_s \,ds$, an application of the stochastic Fubini theorem yields an implicit Volterra equation for $X$ that eliminates the stochastic integral \citep*[Lemma~2.1]{abi2021weak}:
\begin{equation}
\label{eq:hyper_rough_heston}
X_t = V_0\,t + \lambda \,\theta \,\frac{t^{H+3/2}}{(H+1/2)(H+3/2)}  + \int_0^t \, (t-s)^{H - 1/2}\, \left(- \lambda \, X_s + M_s \right)\,ds\,,
\end{equation}
where  $M$ is a martingale with quadratic variation 
\begin{align}\label{eq:quad_M}
    \langle M \rangle_t = \nu^2\, X_t\,.
\end{align} 
When $H > 0\,$, it is given by $M_t := \nu\, \int_0^t\,\sqrt{V_s}\,dW_s\,$.
The absence of a stochastic integral in \eqref{eq:hyper_rough_heston} means that local integrability of the fractional kernel $K_H$ (instead of square-integrability at zero) is sufficient to make sense of the right-hand side of \eqref{eq:hyper_rough_heston}, allowing us to extend the equation to nonpositive Hurst parameters $H \in (-1/2\,,0]\,$. This extension defines the so-called \textit{hyper-rough} process. Such stochastic Volterra equations with locally integrable kernels have appeared  several times in the literature. For $H=0\,$, they describe the local occupation time of a catalytic super-Brownian motion at the catalyst point 0\,, as studied by \citet{dawson1994super,fleischmann1995new}.  More recently, for $H\in (-1/2\,,0]\,$, these equations have emerged as scaling limits of the integrated intensity of nearly unstable heavy-tailed Hawkes processes in \citet{jusselin2020no}. The existence and uniqueness of weak solutions to \eqref{eq:hyper_rough_heston}  were established by \citet{abi2021weak} for a broad class of locally integrable convolution kernels. Furthermore, \citet{jusselin2020no} demonstrated that for $H\in (-1/2\,,0]\,$, the  process $X$  remains continuous but is no longer absolutely continuous, implying that the density process $V$ ceases to exist. However, the solution $X$ to \eqref{eq:hyper_rough_heston} remains a nondecreasing process starting from zero, a crucial property for the well-definedness of the quadratic variation of the martingale $M$ in \eqref{eq:quad_M} in terms of $X$.\\

A natural question emerging in this framework is: 
\begin{center}
{\textit{What happens in the extreme case where roughness increases beyond hyper-rough regimes $H \in (-1/2\,,0]\,$, particularly as $H \to -1/2\,$?}}
\end{center}
Our work aims to elucidate a remarkable transition by establishing a fundamental link between \textit{hyper-rough} processes and \textit{jump} processes via a continuum of the Hurst index $H$.\\

As $H$ approaches $-1/2\,$, we expect the trajectories to become increasingly erratic, given that the fractional kernel $K_H$ tends towards the nonintegrable function $1\,/\,t\,$. A first attempt to understand this transition was made in \citet{abi2024reconciling}, where the authors introduced a {Markovian one-factor proxy} $X^{\varepsilon}$ as a short-time numerical approximation of $X$. Specifically, they consider the integrated process $X^{\varepsilon}_t := \int_0^t \,V^{\varepsilon}_s \,ds$ with 
 $$dV^{\varepsilon}_t = \frac{1}{\varepsilon}\,(\theta - V^{\varepsilon}_t) \,dt + \nu \,\varepsilon^{H-1/2} \,\sqrt{V^{\varepsilon}_t} dW_t\,, \quad \varepsilon > 0\,,$$
which is defined for any $H \in \R\,$. They study the weak limit of $X^{\varepsilon}$ in the Skorokhod $M_1$ topology, as $\varepsilon \to 0\,$. For $H > - 1/2\,$, they obtain a constant variance limit, whereas for $H = -1/2\,$, they prove that $X^{\varepsilon}$ converges to a pure jump Lévy process: the celebrated inverse Gaussian process. This result provides intriguing indirect evidence that hints at a connection between \textit{hyper-rough} processes and \textit{jumps}. Nevertheless, 
the approach relied on a Markovian approximation rather than working directly with the true process $X$, which ruled out establishing a direct connection.\\

The limiting inverse Gaussian (IG) distribution of the Markovian square-root process (with $H = 1/2$) under fast mean-reversion and high volatility was first established by \cite*{mechkov2015fast}, motivated by the long-time behaviour derived in \cite*{forde2011large}. This was later investigated further by \cite{McCrickerd2021}. More recently, an analogous behaviour has been studied for the rough square-root process with fixed $H \in (0\,, 1/2)$ in the work of \cite*{bondi2025vy}.\\

In this work we establish the transition directly on the true \textit{hyper-rough process} $X$, without resorting to a Markovian proxy. Compared to \cite{abi2024reconciling}, our setting is more involved: we work with the {fully non-Markovian and nonsemimartingale} process, analyze the {joint convergence of $(X,M)$} rather than just one factor, and use a refined \textit{topological framework} to rigorously characterize the limiting behavior.\\

To make this transition precise, we show that the couple $(X, M)$ from the \textit{hyper-rough} square-root process in \eqref{eq:hyper_rough_heston} converges weakly to a \textit{jump} process as $H \to -1/2\,$, in a certain topology. Specifically, we establish the weak convergence of $X$ to an inverse Gaussian process and of $M$ towards a compensated version of $X$ (see Theorem~\ref{theorem:weak_convergence}). Since the couple of \textit{hyper-rough} processes $(X, M)$ is continuous and we expect a discontinuous limit, the classical $J_1$ Skorokhod topology on the space of càdlàg functions cannot be used (see \citet{kern2022skorokhod} for an intuitive view on the \cite{skorokhod1956limit} topologies). In order to allow discontinuous limits from a sequence of continuous functions, we use the $M_1$ Skorokhod topology for $X$. However, recent results from \citet{sojmark2023weak} show that for a sequence of continuous local martingales, $M_1$ convergence implies $J_1$ convergence. Thus, it is not possible that $M$ converges in the $M_1$ topology. Instead, we employ a non-Skorokhod topology adapted to semimartingales: the $S$ topology introduced by \citet{jakubowski1997non}. The $M_1$ topology is weaker than the $J_1$ topology, and the $S$ topology is weaker than both, while still being finer than the $L^1$ or the pseudo-path topology of \citet{meyer1984tightness}.\\

Throughout this paper, we fix $T >0$ and we consider a sequence of Hurst coefficients $(H^n)_{n \geq 0}$ such that for any $n \geq 0\,$, $H^n > - 1/2\,$ and $H^n \to -1/2$ as $n$ goes to infinity. We define  the sequence of associated rescaled fractional kernels $(K^n)_{n \geq 0}$ by
\begin{align}\label{eq:Kn}
K^n(t) := \left(H^n + 1/2\right)\,t^{H^n - 1/2}, \quad t\leq T, \quad n \geq 0\,.   
\end{align}
 Thanks to \citet{jusselin2020no} and \citet[Theorem~2.13]{abi2021weak}, for any $n \geq 0\,$, there exists a filtered probability space $\left(\Omega^n, \mathcal{F}^n, (\mathcal{F}^n_t)_{t \leq T}, \mathbb{P}^n\right)$ satisfying the usual conditions, as well as a nondecreasing nonnegative continuous process $X^n$ and a continuous local martingale $M^n$ with quadratic variation 
 \begin{align}\label{eq:Mnquad}
  \langle M^n \rangle = \nu^2 \,X^n   
 \end{align} 
 such that
\begin{equation}
    \label{eq:definition_sequence}
X^n_t = G_0^n(t) + \int_0^t\,(- \lambda \, X^n_s + M^n_s)\, K^n(t-s)\, ds\,, \quad t \leq T\,,
\end{equation}
where $G_0^n$ is defined by
\begin{equation}
    \label{eq:G_0_n_def}
    G_0^n(t) := V_0\, t + \frac{\lambda \,\theta}{H^n + 3/2}\, t^{H^n + 3/2}\,, \quad t \leq T\,.
\end{equation}
Setting 
\begin{align}
    \label{eq:G_0_def}
  G_0(t) := g_0 \,t\,, \quad t \leq T\,,  \quad \text{with} \quad   g_0 := V_0 + \lambda \, \theta\,,  
\end{align}
the sequence $(G_0^n)_{n\geq 0}$ converges to $G_0$ uniformly on $[0, T]\,$. 
\begin{center}
    \textit{We are interested in the weak convergence of the couple $\left(X^n, M^n\right)$ as $n$ goes to infinity.}
\end{center}

In Section \ref{S:intuition}, we give some intuition for the convergence of our processes to inverse Gaussian jump processes using the fact that the fractional kernels in \eqref{eq:Kn} converge weakly to a Dirac measure as $H$ goes to $-1/2$\,. In Section \ref{S:Main}, we state our main results and illustrate the convergence using a specific numerical scheme. Section \ref{S:Proofs} contains the proofs of the announced results. Finally, in Appendices~\ref{A:M1_topology} and \ref{app:S_topology}, we collect the main properties and ideas behind both the $M_1$ and the $S$ topology on the space of càdlàg functions. 

\paragraph*{Notation.} We denote by $\mathbb{D}_T := D\left([0\,,T]\,, \R\right)$ the space of real-valued càdlàg functions on $[0\,,T]\,$. When it is properly defined, we denote the convolution of two functions by 
\begin{equation}
    \label{eq:def_convolution}
    (f * g)(t) := \int_0^t \,f(t-s)\,g(s)\,ds \,, \quad t \leq T\,,
\end{equation}
and extend this notation to 
\begin{equation}
    \label{eq:def_convolution_diff}
    (f * dg)(t) := \int_0^t \, f(t-s)\,dg(s)\,, \quad t \leq T\,,
\end{equation}
for Lebesgue, Stieltjes or Itô integrals.

\section{Intuition for the jump limit}\label{S:intuition}
We start by introducing our limiting process, the inverse Gaussian (IG) Lévy process. We refer the reader to \citet[I.4]{tankov2003financial} and \citet{barndorff2012basics} for more details on the subject.
\begin{definition}[Inverse Gaussian random variable]
    \label{def:IG}
    Let $\mu\,,\lambda > 0\,$. A real-valued random variable $Y$ follows an inverse Gaussian law with parameters $(\mu\,, \lambda)\,$, denoted by $IG\left(\mu\,, \lambda\right)\,$, if its probability density function is given by 
    $$
    f_{\mu,\lambda}(y) = \sqrt{\frac{\lambda}{2 \pi y^3}}\,\exp\left(-\frac{\lambda \, (y - \mu)^2}{2 \mu^2 y}\right)\, \1_{y > 0}\,, \quad y \in \R\,.
    $$
\end{definition}
\noindent
In particular, its mean is $\mu$ and its variance is $\frac{\mu^3}{\lambda}\,$. Moreover, its characteristic function is 
$$
\mathbb{E}\left[\exp(iu\,Y)\right] = \exp\left(\frac{\lambda}{\mu}\left[1 - \sqrt{1 - 2\,\frac{\mu^2}{\lambda}\,iu}\right]\right)\,,\quad u \in \R\,.
$$

\begin{definition}[Inverse Gaussian process]
    \label{def:IG_process}
    We say that $(Y_t)_{0\leq t \leq T}$ is an IG process with parameters $(\mu\,, \lambda)$ if it is a Lévy process with càdlàg samples paths, almost surely starting from $0\,$, with Lévy exponent given by 
    $$
    \phi(u) := \frac{\lambda}{\mu}\left[1 - \sqrt{1 - 2\,\frac{\mu^2}{\lambda}\,iu}\right]\,, \quad u \in \R\,.
    $$
\end{definition}
\noindent
This gives the following marginal characteristic function
$$
\mathbb{E}\left[\exp\left(iu\,Y_t\right)\right] = \exp(\phi(u)\,t)\,, \quad  u \in \R\,, \quad t \leq T\,.
$$
Thus, for a fixed $t\,$, we have $Y_t \sim IG\left(\mu\,t\,, \lambda \,t^2\right)\,$. Moreover, the exponent can be rewritten as
$$
\phi(u) = \int_{\R_+}\, \left(e^{iux} - 1 \right) \, \nu(dx)\,,
$$
where the Lévy measure is
$$
\nu(dx) = \sqrt{\frac{\lambda}{2\pi x^3}}\,\exp\left(-\frac{\lambda \, x}{2 \mu^2}\right)\, \1_{x > 0}\, dx\,,
$$
showing that $(Y_t)_{0\leq t\leq T}$ is a nondecreasing pure jump Lévy process. Moreover, since $\nu$ is nonintegrable in $0\,$, it is an infinite activity process, meaning it jumps an infinite number of times on every interval.

Additionally, we have the following representation for the IG process as hitting times of a Brownian motion with drift, see for instance \citet[Example~1.3.21]{applebaum2009levy}.
\begin{lemma}
    \label{lemma:IG_representation}
    Let $a\,,b\,,c > 0\,$, for any $t\,$, we define $Y_t$ as the first-hitting time of the drifted Brownian motion $(a\,s + b\,W_s)_s$ at threshold $c\,t$\,, that is,
    $$
Y_t := \inf\{s \geq 0\,:\, a\,s + b\,W_s = c\,t \}\,.
    $$
    Then, $(Y_t)_t$ is an inverse Gaussian process with parameters $\left(\frac{c}{a}\,, \frac{c^2}{b^2} \right)\,$.
\end{lemma}
The convergence of the hyper-rough processes $\left(X^n\,, M^n\right)$ in \eqref{eq:Mnquad}-\eqref{eq:definition_sequence} to a Lévy process of inverse Gaussian type can be intuitively understood with the following lemma.
\begin{lemma}
    \label{lemma:weak_cv_to_dirac}
{For any $n \geq 0\,$, define the measure $\mu^n(dt) := K^n(t)\,dt\,$, where $K^n$ is given by \eqref{eq:Kn}\,. Then, for any $f \in \mathbb{D}_T\,$,
    $$
    \int_0^t \, f(t -s)\,\mu^n(ds) \overset{n \to \infty}{\longrightarrow} f(t)\,, \quad t \in (0\,,T] \setminus Disc(f)\,,
    $$
    where $Disc(f) := \left\{ 0 \leq t \leq T\,:\,f(t-) \neq f(t)\right\}$ is the discontinuity set of $f\,$.
    In particular, if we denote by $\delta_0$ the Dirac measure in $0\,$, we have the weak convergence
     $$\mu^n \overset{n \to \infty}{\implies} \delta_0\,, 
    $$ in the space of finite nonnegative measures on $\left([0\,,T]\,, \mathcal{B}_{[0\,,T]}\right)\,$.}
\end{lemma}
\begin{proof}
    Let $f \in \mathbb{D}_T\,$, $t \in (0\,,T] \setminus Disc(f)$ and $\varepsilon > 0\,$. There exists $0 < \eta < t$ such that
$$
\left|f(t-s) - f(t)\right| \leq \varepsilon\,, \quad  s \leq \eta\,.
$$
Computing
$$
\int_{t_0}^{t_1}\, \mu^n(ds) = t_1^{H^n + 1/2} - t_0^{H^n + 1/2}\,, \quad t_0 \leq t_1 \leq T\,, \quad n \geq 0\,,
$$
we get for any $n \geq 0\,$,
\begin{align*}
    \left|\int_0^t\,f(t-s)\, \mu^n(ds) - f(t)\right| &\leq \int_0^t \, \left|f(t-s) - f(t) \right|\, \mu^n(ds) + \left|f(t)\,\left(t^{H^n + 1/2} - 1\right) \right| \\
    &\leq \varepsilon\, \eta^{H^n + 1/2} + 2\,\|f\|_{\infty} \, \left(t^{H^n + 1/2} - \eta^{H^n + 1/2}\right) \\
    &\quad+ \left|f(t)\,\left(t^{H^n + 1/2} - 1\right) \right| \,.
\end{align*}
Finally, since $H^n + 1/2$ goes to $0$ as $n \to + \infty\,$, we have 
$$
\int_0^t\, f(t-s) \, \mu^n(ds) \overset{n \to \infty}{\longrightarrow} f(t) = \int_{[0\,,t]}\, f(t-s)\, \delta_0(ds)\,. \qedhere
$$
\end{proof}
\begin{remark}
Note that for a decreasing sequence $\left(\alpha^n\right)_{n \geq 0}$ such that $\alpha^n \overset{n \to \infty}{\longrightarrow} \alpha > -1\,$, one cannot find a scaling $\left(\beta^n\right)_{n \geq 0}$ such that $\left(\beta^n\,t^{\alpha^n}\,dt\right)_{n \geq 0}$ converges weakly to $\delta_0\,$. Thus, Lemma \ref{lemma:weak_cv_to_dirac} is specific to the limit $H^n - 1 / 2 \to -1\,$.  
\end{remark}

The inverse Gaussian structure emerges from the limiting equation and the interpretation of the inverse Gaussian process as the first-passage time of a drifted Brownian motion. To see this, we rewrite \eqref{eq:definition_sequence} as  
$$
X^n_t = G_0^n(t) + \int_0^t\,\left(-\lambda \, X^n_{t-s} + \nu\,W_{X^n_{t-s}}\right)\, \mu^n(ds)\,,
$$
where $\mu^n$ is defined in Lemma \ref{lemma:weak_cv_to_dirac} and $W$ is a standard Brownian motion, coming from the Dambis-Dubins-Schwarz theorem applied on $M^n$ (even though this would require additional assumptions on the probability space). In the limit, as $n \to \infty\,$, we heuristically expect an equation of the form  
\begin{equation}
    \label{eq:heuristics_limit_equation}
X_t = g_0\,t - \lambda\,X_t + \nu\,W_{X_t}
\end{equation}
due to Lemma~\ref{lemma:weak_cv_to_dirac}. A nondecreasing solution $X$ to this equation is the inverse Gaussian process with parameters $\left(\frac{g_0}{1 + \lambda}\,, \frac{g_0^2}{\nu^2}\right)$ (see Definition~\ref{def:IG_process} and Lemma~\ref{lemma:IG_representation}).  Furthermore, the limiting equation immediately yields $M$ as the compensated process
$$
M = (1 + \lambda)\, X - G_0\,.
$$

\begin{remark}
    Lemma \ref{lemma:weak_cv_to_dirac}, as well as all the results of the present paper, are true for any kernel $K^n(t) = \alpha^n\,t^{H^n - 1/2}$ with a scaling $\alpha^n \sim H^n + 1/2$ as $H^n$ goes to $-1/2\,$. In particular, this is the case for the historical scaling in $1 \,/\, \Gamma\left(H^n + 1/2\right)\,$ of the rough square-root process \eqref{eq:rough_heston}. This scaling was first chosen in \citet{mandelbrot1968fractional} and \citet{levy1953random} to define two different types of fractional Brownian motions, making sure that the definition coincides with the \textit{Riemann-Liouville fractional integral} given by
\begin{equation}
    \label{eq:fractional_integral}
I^{\alpha}\,f(x) := \frac{1}{\Gamma(\alpha)}\,\int_0^x \,(x - t)^{\alpha - 1}\, f(t)\,dt\,, \quad \alpha > 0\,.
\end{equation}
The constant $1 \,/\, \Gamma(\alpha)$ ensures that it satisfies the fundamental integration properties
$$
I^{\alpha}\left(I^{\beta}\,f\right) = I^{\alpha + \beta}\,f\,, \quad \frac{d}{dx}\,I^{\alpha + 1}\,f(x) = I^{\alpha}\,f(x)\,,
$$
as well as the fact that for $\alpha \in \N\,$, \eqref{eq:fractional_integral} coincides with Cauchy's iterated integral. Using this scaling, \eqref{eq:definition_sequence} can be rewritten as
$$
X^n_t = G_0^n(t) + I^{H^n + 1/2}\,\left(- \lambda \, X^n + M^n\right)(t)\,, \quad t \leq T\,.
$$
Therefore, heuristically, in the limit $H^n \to -1/2\,$, we obtain the fractional integral of order $0\,$, that is, the identity, which is consistent with \eqref{eq:heuristics_limit_equation}. 
\end{remark}

\section{Main results}\label{S:Main}

\subsection{Weak convergence to an inverse Gaussian process}

Our main result establishes the weak convergence of the entire sequence $\left(X^n, M^n\right)_{n \geq 0}$ (not just along a subsequence), as defined in \eqref{eq:Mnquad}-\eqref{eq:definition_sequence}, towards inverse Gaussian processes. Additionally, we provide a type of almost sure Skorokhod representation along a subsequence.  The convergence holds in the product topology $M_1 \otimes S$ on the space $\mathbb{D}_T^2\,$. The $M_1$ and $S$ topologies are recalled in  Appendices~\ref{A:M1_topology} and \ref{app:S_topology}.

\begin{theorem}
    \label{theorem:weak_convergence}
    Let $Y$ be an inverse Gaussian Lévy process with parameters $\left(\frac{g_0}{1 + \lambda}\,, \frac{g_0^2}{\nu^2} \right)\,$, in the sense of Definition \ref{def:IG_process}.
    We have the weak convergence 
    $$\left(X^n\,, M^n\right) \overset{n \to \infty}{\implies} \left(Y\,, (1 + \lambda)\,Y - G_0\right)$$
     on the space $\mathbb{D}_T^2$ endowed with the product topology $M_1 \otimes S\,$, where the function $G_0$ is given by \eqref{eq:G_0_def}.\\
    \noindent
    Furthermore, there exist a subsequence $\left(X^{n_{l}}\,, M^{n_{l}}\right)_{l \geq 0}$ and càdlàg processes $\left(\hat{X}^l\,, \hat{M}^l\right)_{l \geq 0}\,$, $\hat{X}$ and $\hat{M}$ defined on the probability space $\left([0\,,1]\,, \mathcal{B}_{[0\,, 1]}\,, dt\right)\,$, such that:
    \begin{enumerate}
        \item For any $l \geq 0\,$, $\left(X^{n_{l}}\,, M^{n_{l}}\right) \sim \left(\hat{X}^l\,, \hat{M}^l\right)\,$.
        \item For any $\omega \in [0\,, 1]\,$, $\left(\hat{X}^l(\omega)\,, \hat{M}^l(\omega)\right) \overset{M_1 \otimes S}{\longrightarrow} \left(\hat{X}(\omega)\,, \hat{M}(\omega)\right)\,$.
        \item $\left(\hat{X}\,, \hat{M}\right) \sim \left(Y\,, (1 + \lambda)\, Y - G_0\right)\,$, so that almost surely
        $$
        \hat{M}_t = (1 + \lambda)\, \hat{X}_t - g_0\,t\,, \quad t \leq T\,.
        $$
    \end{enumerate}
\end{theorem}
\begin{proof}
    The proof is given in Section \ref{subsec:proof_weak_convergence}.
\end{proof}

The main tools for the proof of this theorem are given in the following subsections: Lemma \ref{lemma:tightness} gives tightness of the sequence of processes and Lemma \ref{lemma:characteristic_convergence} will help us identify the limiting law using the characteristic functional. Nevertheless, since $(\mathbb{D}_T\,, S)$ is not metrizable (see Proposition \ref{proposition:S_properties}), the Prokhorov theorem as well as the a.s. Skorokhod representation theorem cannot be readily applied. Therefore, the identification of the limiting law with the convergence of the characteristic functional given by Lemma \ref{lemma:characteristic_convergence} cannot be achieved with classical methods. To solve these issues, we need to use specific tools such as the convergence $\overset{*}{\longrightarrow}_{\mathcal{D}}$ from Definition \ref{definition:cv_D}, which can be seen as an \textit{almost sure compact-valued representation} for weak convergence in the $S$ topology, and the associated variant of Skorokhod's representation theorem for general topological spaces admitting a countable family of functionals separating their points given by Proposition \ref{proposition:skorokhod_representation_non_metrizable}. Note that even though tightness of the sequence $\left(X^n \,, M^n\right)_{n \geq 0}$ for the product topology $M_1 \otimes S$ simply corresponds to tightness of $\left(X^n\right)_{n \geq 0}$ for $M_1$ and of $\left(M^n\right)_{n \geq 0}$ for $S\,$, some subtleties still need to be addressed when applying Proposition \ref{proposition:skorokhod_representation_non_metrizable} to $\left(\mathbb{D}_T^2\,, M_1 \otimes S\right)\,$. Specifically, one needs to check that this topological space admits a countable family of continuous functionals separating its points, which is done at the beginning of the proof of Theorem \ref{theorem:weak_convergence}, as well as ensuring measurability of the couple $\left(X^n\,,M^n\right)$ with respect to the Borel $\sigma$-field generated by the product topology (see Remark \ref{remark:product_sigma_field}).\\

To illustrate Theorem~\ref{theorem:weak_convergence}, we implement the iVi scheme of \citet{abijaber2025simulating}, 
 which simulates the integrated square-root process and its associated martingale using inverse Gaussian increments. This scheme provides a clear visual representation of the convergence stated in Theorem~\ref{theorem:weak_convergence}.   We set the parameters to $V_0 = 0.1\,$, $\lambda = 10\,$, $\theta = 0.1\,$, $\nu = 1\,$, and $T = 1\,$.

Figure~\ref{fig:trajectory_cv} displays simulated trajectories of $X^n$, $M^n$, and $G_0^n + M^n - (1 + \lambda)\,X^n$ as $H^n$ approaches $-1/2\,$, using the same random seed. The transition in the behavior of $X$ is clearly visible: it evolves from being absolutely continuous for $H > 0$ to continuous but not absolutely continuous for $-1/2 < H \leq 0\,$, and progressively becomes discontinuous as $H$ nears $-1/2\,$.   Furthermore, the convergence to the limiting inverse Gaussian (IG) process (black dotted curve) is remarkably visible. Similarly, $M$ develops a jump behavior as $H \to -1/2\,$, aligning more closely with the limiting shifted IG process. Finally, we observe that the relation $(1 + \lambda)X = G_0 + M\,$, which does not hold for $H \gg -1/2\,$, is more closely satisfied as $H$ approaches $-1/2\,$. At each time step of the iVi scheme, we sample from the inverse Gaussian distribution using standard normal and uniform random variables, the limiting inverse Gaussian process and its compensated version shown in Figure \ref{fig:trajectory_cv} are simulated with the same random variables. Note that these processes theoretically have infinitely many jumps on any interval, even though this is not visible on the plots.

\begin{figure}[h]
    \centering
    \includegraphics[height = 14cm, width = 13cm]{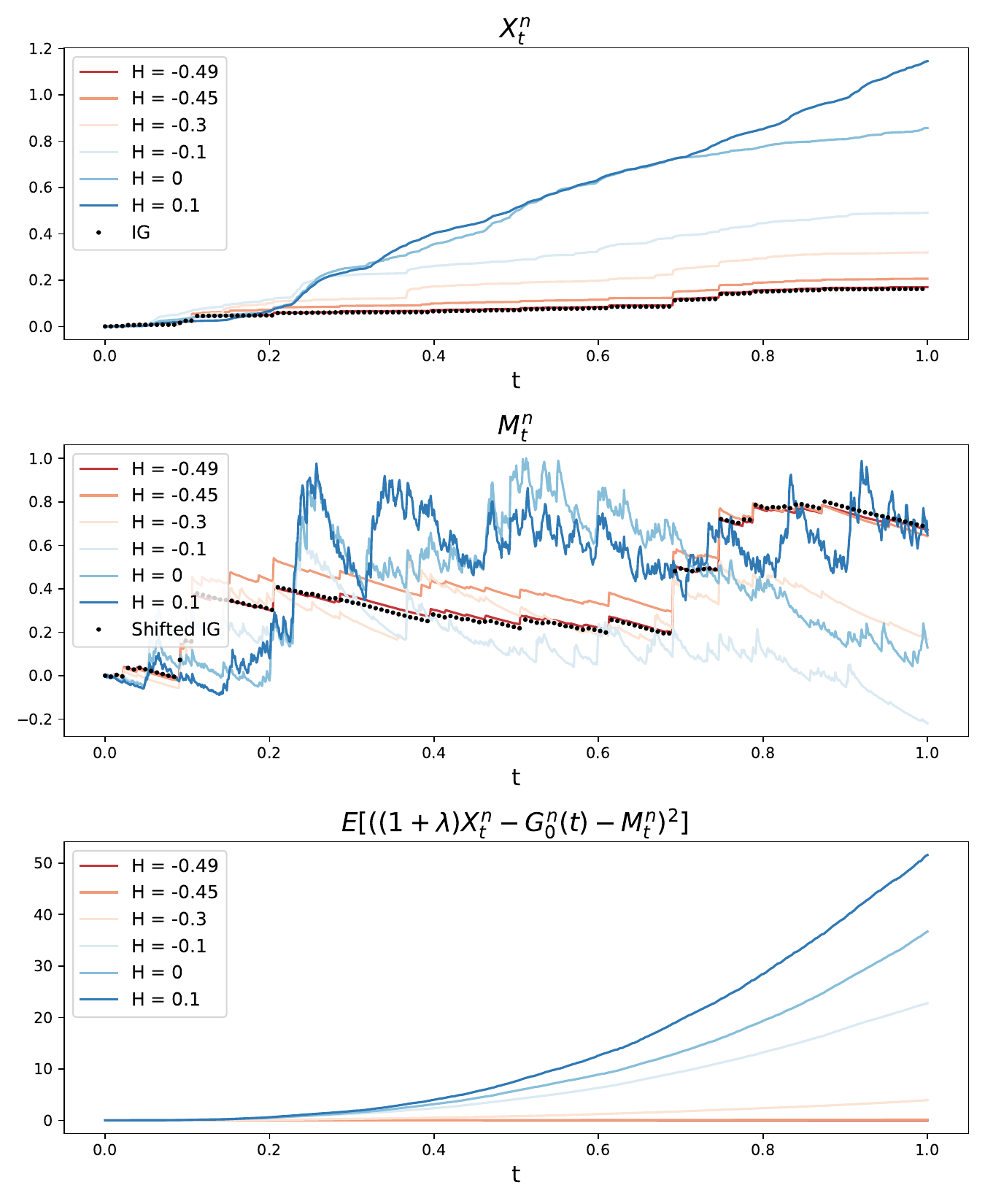}
    \caption{Convergence of $(X^n\,, M^n)_{n\geq 0}$ to a jump process as $H^n$ goes to $-1/2\,$. \textbf{Top:} trajectories of $\left(X^n\right)_{n \geq 0}$ (plain curves) and trajectory of the IG process $Y$ with parameters $\left(\frac{g_0}{1 + \lambda}\,, \frac{g_0^2}{\nu^2}\right)$ (black dotted curve) using the same random seed. \textbf{Middle:} trajectories of $\left(M^n\right)_{n \geq 0}$ (plain curves) and trajectory of the shifted IG process $(1 + \lambda)\, Y - G_0$ (black dotted curve) using the same random seed. \textbf{Bottom:} Monte Carlo estimator of $\mathbb{E}\left[\left((1 + \lambda)\,X^n_t - G_0^n(t) - M^n_t\right)^2\right]$ computed with 100 trajectories.}
    \label{fig:trajectory_cv}
\end{figure}

Figure \ref{fig:marginal_cv} shows the convergence of the marginal distribution of $\left(X^n_T\,, M^n_T\right)$ to those of $\left(Y_T\,, (1 + \lambda)\,Y_T - g_0\,T\right)$ where $Y_T \sim IG\left(\frac{g_0\,T}{1 + \lambda}\,, \frac{g_0^2\,T^2}{\nu^2}\right)\,$ (see Definition \ref{def:IG}), as $H^n$ goes to $-1/2\,$. We plot the empirical density of $X^n_T$ as well as $M^n_T\,$, and the joint characteristic function $\mathbb{E}\left[\exp\left(u\,X^n_T + v\,M^n_T\right)\right]\,$.

\begin{figure}[h]
    \centering
    \hspace{-0.5cm}
    \includegraphics[scale = 0.28]{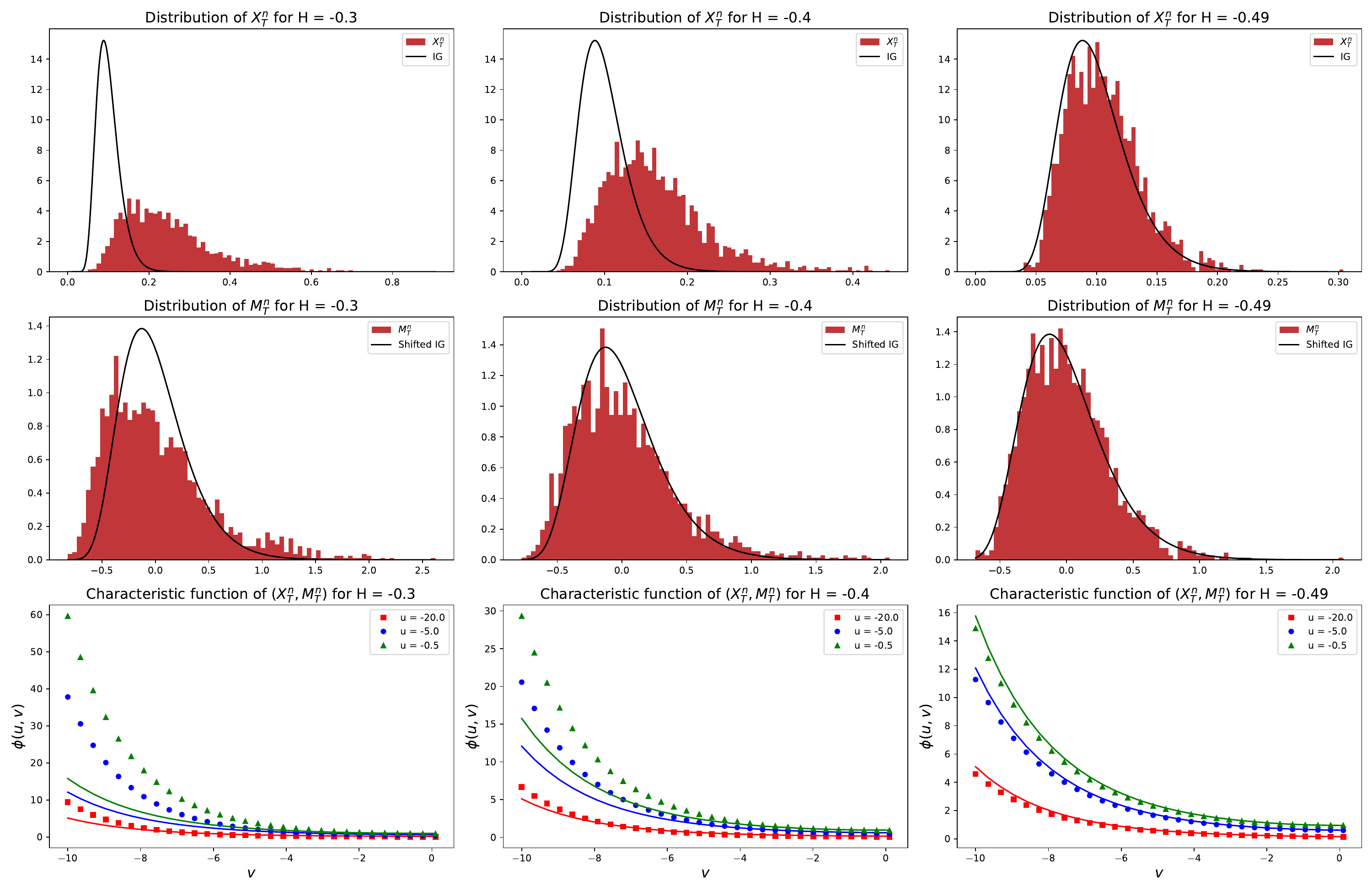}
    \caption{\textbf{Top row:} Empirical density of $X^n_T$ for different values of $H^n\,$, compared with the theoretical limiting inverse Gaussian distribution (black line). \textbf{Middle row:} Empirical density of $M^n_T$ for different values of $H^n\,$, compared with the theoretical limiting shifted inverse Gaussian distribution (black line). \textbf{Bottom row:} Empirical joint characteristic function $\mathbb{E}\left[\exp\left(u\,X^n_T + v\,M^n_T\right)\right]$ (dotted curves), compared with the theoretical limiting inverse Gaussian - shifted inverse Gaussian characteristic function (plain lines), as a function of $v\,$, for different values of $u\,$.}
    \label{fig:marginal_cv}
\end{figure}

\subsection{Estimates}
We derive the following estimates on the sequence $(X^n\,, M^n)_{n\geq 0}\,$.

\begin{lemma}
    \label{lemma:estimates}
    For any $n \geq 0\,$, we have: 
    \begin{enumerate}
        \item $M^n$ is a true $\left(\mathcal{F}^n_t\right)_{t \leq T}$-martingale such that 
        $$\mathbb{E}\left[\left(M^n_t\right)^2\right] = \nu^2 \, \mathbb{E}\left[X_t^n\right] < + \infty\,, \quad t \leq T\,.
        $$
        \item $\mathbb{E}\left[X_t^n\right] \leq G_0^n(t)\,, \quad t \leq T\,$.
        \item $\mathbb{E}\left[X_{t}^n - X_s^n\right] \leq G_0^n(t) - G_0^n(s)\,, \quad  s \leq t \leq T\,$.
    \end{enumerate}
\end{lemma}
\begin{proof}
    The proof is given in Section \ref{subsec:proof_estimates}.
\end{proof}

\subsection{Tightness}
\label{subsec:tightness}
For the proof of Theorem \ref{theorem:weak_convergence}, we need tightness of the sequence $\left(X^n\,, M^n\right)_{n \geq 0}$ for the product topology $M_1 \otimes S\,$, which is given by the following lemma.
\begin{lemma}
    \label{lemma:tightness}
    The sequence $(X^n\,, M^n)_{n \geq 0}$ is tight in $\left(\mathbb{D}_T^2\,, M_1 \otimes S\right)\,$.
\end{lemma}
\begin{proof}
    The proof is given in Section \ref{subsec:proof_tightness}.
\end{proof}

We provide here some justification for the choice of the topology $M_1 \otimes S\,$. As mentioned in the Introduction, the usual $J_1$ Skorokhod topology does not allow a sequence of continuous functions to converge to a discontinuous one. However, this is possible in the $M_1$ topology, which still remains quite informative (stronger than the Skorokhod $M_2$ topology and than the $L^p$ topologies for $0 < p < + \infty$). Moreover, the conditions for tightness in the $M_1$ topology are fairly simple for a sequence of nondecreasing processes. It was already used in \citet{abi2024reconciling}, which inspired the present work. The remaining question is therefore the choice of the $S$ topology for the sequence $\left(M^n\right)_{n \geq 0}\,$. Recent results from \citet[Corollary~4.5]{sojmark2023weak} show that a sequence of continuous local martingales which converges weakly for the $M_1$ topology must converge for the $J_1$ topology\footnote{The result requires that the sequence of local martingales has so-called \textit{good decompositions} (see \citet[Definition~3.3]{sojmark2023weak}), which is the case if they are continuous.}, which cannot be the case in our framework. Hence, it is not possible to prove $M_1$-tightness of $\left(M^n\right)_{n \geq 0}\,$. Furthermore, the $S$ topology is particularly suited for martingales (see Proposition \ref{proposition:S_tightness_supermartingale}) while still providing strong stability properties, as in Proposition \ref{proposition:S_martingale_stability}. 

\subsection{Convergence of the characteristic functional}
Let $f\,,h : [0\,,T] \to \R$ be continuous. According to \citet[Theorem~2.2, Theorem~2.5, Example~2.4]{abi2021weak}, for any $n \geq 0\,$, the joint characteristic functional of $\left(X^n\,, M^n\right)$ is given by

\begin{equation}
    \label{eq:char_func_affine}
\mathbb{E}\left[\exp\left(i \int_0^T\!\!  f(T-t) dX^n_t + i  \int_0^T \!\!h(T-t) dM^n_t\right) \right] = \exp\left(\int_0^T \!\!F\left(T-t, \Psi^n(T-t)\right) dG_0^n(t)\right),
\end{equation}

where
$$
F(t, u) := i\,f(t) - \frac{\nu^2}{2}\, h^2(t) + \left(i \nu^2\, h(t) - \lambda \right)\, u + \frac{\nu^2}{2}\,u^2\,, \quad t \leq T \,, \quad u \in \mathbb{C}\,,
$$
and $\Psi^n : [0\,,T] \to \mathbb{C}$ is the unique continuous solution, with nonpositive real part, to the Riccati-Volterra equation
$$
\Psi^n(t) = \int_0^t\, F\left(s\,,\Psi^n(s)\right)\, K^n(t-s)\, ds\,, \quad  t \leq T\,.
$$
We have the following convergence of the characteristic functional.
\begin{lemma}
    \label{lemma:characteristic_convergence}
    For any $t \in [0\,, T]\,$, let $\Psi(t)$ be the solution, with nonpositive real part, to
    \begin{equation}
        \label{eq:def_Psi}
    \Psi(t) = F(t\,, \Psi(t))\,.
    \end{equation}
    Then $\Psi$ is continuous, and we have the convergence
    $$
    \int_0^T\, F\left(T-t\,, \Psi^n(T-t) \right)\, dG_0^n(t) \overset{n \to \infty}{\longrightarrow} \int_0^T\, F(T-t\,, \Psi(T-t))\, dG_0(t)  = g_0\, \int_0^T\, \Psi(t)\, dt\,,
    $$
    where $g_0$ is defined in \eqref{eq:G_0_def}\,.
\end{lemma}
\begin{proof}
    The proof is given in Section \ref{subsec:proof_characteristic_convergence}.
\end{proof}
\noindent
\begin{remark}
    \label{remark:Psi_is_IG}
    Note that $\Psi$ defined in \eqref{eq:def_Psi} has the explicit expression
$$
\Psi(t) = -i\,h(t) + \frac{1+\lambda}{\nu^2}\, \left(1 - \sqrt{1 - 2 i \, \frac{\nu^2}{(1+\lambda )^2}\,\left(f(t) + (1+ \lambda)\, h(t) \right)} \right)\,, \quad t \leq T\,,
$$
where we take the principal branch of the square-root, so that if $Y$ is an inverse Gaussian process with parameters $\left(\frac{g_0}{1 + \lambda}\,, \frac{g_0^2}{\nu^2}\right)\,$ in the sense of Definition \ref{def:IG_process}, we have
$$
\exp\left(g_0\,\int_0^T \, \Psi(T-t)\,dt\right) = \mathbb{E}\left[\exp\left(i\, \int_0^T\, \left(f(T-t) + (1+\lambda)\,h(T-t) \right)\, dY_t - i \, g_0\,\int_0^T\,h(T-t)\,dt\right) \right]\,,
$$
using the expression of the characteristic function of the integral of a left-continuous function with respect to a L\'evy process given in \citet[Lemma~15.1]{tankov2003financial}.
\end{remark}

\section{Proofs}\label{S:Proofs}
\subsection{Proof of Theorem \ref{theorem:weak_convergence}}
\label{subsec:proof_weak_convergence}
Let $Y = \left(Y_t\right)_{0 \leq t \leq T}$ be an IG process with parameters $\left(\frac{g_0}{1 + \lambda}\,, \frac{g_0^2}{\nu^2}\right)$ (see Definition \ref{def:IG_process}).

Our proof is based on the variant of the Skorokhod representation theorem given by Proposition \ref{proposition:skorokhod_representation_non_metrizable}, applied to the topological space $\left(\mathbb{D}_T^2\,, M_1 \otimes S\right)\,$. Consider any subsequence $\left(X^{n_k}\,, M^{n_k}\right)_{k \geq 0}\,$, it is $M_1 \otimes S$-tight thanks to Lemma \ref{lemma:tightness}, so that we only need to construct a countable family of continuous functionals, separating points of $\left(\mathbb{D}_T^2\,, M_1 \otimes S\right)\,$. Since $(\mathbb{D}_T\,, M_1)$ is Polish, there exists a countable sequence $\left(f^1_n\right)_{n \geq 0}$ of $M_1$-continuous functionals that separate its points. For example, take $\left(x_n\right)_{n \geq 0}$ a countable dense family (from separability) and define $\left(f_n^1\right)_{n \geq 0} := \left(d_{M_1}\left(\,\cdot\,, x_n\right)\right)_{n \geq 0}$ where $d_{M_1}$ is the metric of the $M_1$ topology. Moreover, from Proposition \ref{proposition:S_properties}, there also exists such a family $\left(f^2_n\right)_{n \geq 0}$ for the $S$ topology. Defining on $\mathbb{D}_T^2\,$, for any $n \geq 0\,$, $\tilde{f}^1_n : (x, y) \mapsto f^1_n(x)$ and $\tilde{f}^2_n : (x,y) \mapsto f^2_n(y)\,$, the functionals $\left(\tilde{f}^1_n\,, \tilde{f}^2_n\right)_{n \geq 0}$ are $M_1 \otimes S$-continuous and form a countable family separating points in $\left(\mathbb{D}_T^2\,, M_1 \otimes S\right)\,$. Thus, we can apply Proposition \ref{proposition:skorokhod_representation_non_metrizable} to $\left(X^{n_k}\,, M^{n_k}\right)_{k \geq 0}\,$, based on the measurability check given by Remark \ref{remark:product_sigma_field}. We conclude that there exists a further subsequence $\left(X^{n_{k_l}}\,, M^{n_{k_l}}\right)_{l \geq 0}\,$, as well as càdlàg processes $\left(\hat{X}^l\,, \hat{M}^l\right)_{l \geq 0}\,$, $\hat{X}$ and $\hat{M}$ defined on $\left([0\,, 1]\,, \mathcal{B}_{[0\,,1]}\,, dt\right)\,$ such that (see Proposition \ref{proposition:skorokhod_representation_non_metrizable}):
\begin{itemize}
    \item[$\bullet$] For any $l\geq 0\,$, $\left(X^{n_{k_l}}\,, M^{n_{k_l}} \right) \sim \left(\hat{X}^l\,, \hat{M}^l\right)\,$.
    \item[$\bullet$] For any $\omega \in [0\,,1]\,$, $\left(\hat{X}^l(\omega)\,, \hat{M}^l(\omega)\right) \overset{M_1 \otimes S}{\longrightarrow} \left(\hat{X}(\omega)\,, \hat{M}(\omega)\right)\,$.
    \item[$\bullet$] For any $\varepsilon > 0\,$, there exists an $M_1 \otimes S$-compact $K_{\varepsilon} \subset \mathbb{D}_T^2$ such that 
        $$
        \mathbb{P}\left( \left\{\left(\hat{X}^l \,, \hat{M}^l\right) \in K_{\varepsilon}\,, \, l \geq 0 \right\} \right) > 1 - \varepsilon\,.
        $$
\end{itemize}
If we manage to show that $\left(\hat{X}\,, \hat{M}\right) \sim \left(Y\,, (1+ \lambda)\, Y - G_0\right)$ independently of the chosen subsequence $(n_k)_{k \geq 0}\,$, then Theorem \ref{theorem:weak_convergence} is proved. To see this, we consider the two parts of the theorem separately:
\begin{itemize}
    \item[$\bullet$] \textbf{Weak convergence:} for any $f : \mathbb{D}_T^2 \to \R\,$, $M_1 \otimes S$-continuous and bounded, consider the sequence $\left(u_n\right)_{n \geq 0} := \left(\mathbb{E}\left[f\left(X^n\,, M^n\right)\right]\right)_{n \geq 0}\,$. It is bounded due to the boundedness of $f\,$. Moreover, for any convergent subsequence $\left(u_{n_k}\right)_{k \geq 0}\,$, its limit must be $\mathbb{E}\left[f\left(Y \,, (1 + \lambda)\,Y - G_0\right)\right]\,$. This proves $\left(X^n\,, M^n\right) \overset{M_1 \otimes S}{\implies} \left(Y\,, (1 + \lambda)\, Y - G_0\right)\,$.
    \item[$\bullet$] \textbf{Almost sure representation:} taking the subsequence $n_k = n\,$, we have the existence of a subsequence $\left(X^{n_l}\,, M^{n_l}\right)_{l \geq 0}\,$, as well as càdlàg processes $\left(\hat{X}^l\,, \hat{M}^l\right)_{l \geq 0}\,$, $\hat{X}$ and $\hat{M}$ defined on $\left([0\,,1]\,, \mathcal{B}_{[0\,,1]}\,, dt\right)\,$, satisfying 1, 2 and 3 from Theorem \ref{theorem:weak_convergence}. 
\end{itemize}
The problem is thus reduced to the identification of the law of $\left(\hat{X}\,, \hat{M}\right)$ using the convergence of the characteristic functional given by Lemma \ref{lemma:characteristic_convergence}.

First, we remark that we have $M^{n_{k_l}} \overset{*}{\longrightarrow}_{\mathcal{D}} \hat{M}$ in the sense of Definition \ref{definition:cv_D}, which corresponds to a form of \textit{almost sure compact-valued representation} for the convergence in the $S$ topology. Moreover, from Lemma \ref{lemma:estimates}, for any $l \geq 0\,$, $M^{n_{k_l}}$ is an $\left(\mathcal{F}^{n_{k_l}}_t\right)_{t\leq T}$-martingale, and
    $$
    \sup_n \, \mathbb{E}\left[\sup_{0 \leq t \leq T}\, \left| M^n_t\right|^2\right] \leq 4\,\nu^2\,\sup_n\, \mathbb{E}\left[X^n_T\right] \leq 4\, \nu^2\, \sup_n \,G_0^n(T) < + \infty\,,
    $$
    using the BDG inequality \citep[Theorem~IV.54]{protter2005stochastic}.
    Therefore, $\hat{M}$ is a square-integrable martingale with respect to its own filtration, in virtue of Proposition \ref{proposition:S_martingale_stability} which gives a form of stability of the martingale property for the $S$ topology. Thus, we can define Itô integrals with respect to it. Thanks to Proposition \ref{proposition:M_1_local_uniform} and the discussion below, $M_1$-convergence implies pointwise convergence outside of a countable subset of $[0\,,T)$ (the set of discontinuities of the limit), hence the $M_1$-limit of a sequence of nondecreasing functions is also nondecreasing, by right-continuity. Thus, $\hat{X}$ is almost surely nondecreasing, so that we can define Stieltjes integrals with respect to its paths. Additionally, since $\hat{X}$ cannot have a discontinuity at zero, we obtain $\hat{X}_0 = \lim\limits_{l \to \infty}\,\hat{X}^l_0 = 0\,.$

Second, for any continuous $f\,, h : [0\,,T] \to \R\,$, combining \eqref{eq:char_func_affine} and Lemma \ref{lemma:characteristic_convergence} gives us
    $$
    \mathbb{E}\left[\exp\left(i\, \left[f * d\hat{X}^l\right](T) + i \, \left[h * d\hat{M}^l\right](T)\right) \right] \overset{l \to \infty}{\longrightarrow} \exp\left(g_0\,\int_0^T\, \Psi(T-t)\, dt\right)\,,
    $$
    where we used the convolution notation introduced in \eqref{eq:def_convolution_diff} and $\Psi$ is given explicitly in Remark \ref{remark:Psi_is_IG}. In the remainder of this proof, we write $x * y$ instead of $(x * y)(T)$ since all integrals are taken between $0$ and $T\,$.
    For $f$ and $h$ of class $\mathcal{C}^1\,$, we can write
    $$
    \int_0^T \, f(T-t)\, d\hat{X}^l_t = f(0)\, \hat{X}^l_T + \int_0^T\,f'(t)\, \hat{X}^l_{T-t}\, dt\,, \quad l \geq 0\,,
    $$
    which converges almost surely to $f(0)\, \hat{X}_T + \int_0^T \, f'(t)\, \hat{X}_{T-t}\, dt$ from Proposition \ref{proposition:M_1_local_uniform} and Remark \ref{remark:M_1_cv_implies_bounded}, using the dominated convergence theorem.
    Similarly,
    $$
    \int_0^T\, h(T-t)\, d\hat{M}^l_t = h(0)\, \hat{M}^l_T + \int_0^T\, h'(t)\,\hat{M}^l_{T-t}\, dt \overset{l \to \infty}{\longrightarrow} h(0) \, \hat{M}_T + \int_0^T \,h'(t)\, \hat{M}_{T-t}\, dt
    $$
    almost surely, using Remark \ref{remark:S_cv_pointwise_at_T} and the last point of Proposition \ref{proposition:S_properties}. Since $\hat{X}$ is nondecreasing and $\hat{M}$ is a martingale, we have proved that, almost surely,
    $$
    \int_0^T \, f(T-t)\, d\hat{X}^l_t \overset{l \to \infty}{\longrightarrow} \int_0^T \, f(T-t)\, d\hat{X}_t\,,
    $$
    $$
    \int_0^T \, h(T-t)\, d\hat{M}^l_t \overset{l \to \infty}{\longrightarrow} \int_0^T\, h(T-t)\, d\hat{M}_t\,.
    $$
    By boundedness and continuity of $x \mapsto e^{ix}\,$, we get
    $$
    \mathbb{E}\left[\exp\left(i\, \left[f * d\hat{X}^l\right] + i \, \left[h * d\hat{M}^l\right]\right) \right] \overset{l \to \infty}{\longrightarrow} \mathbb{E}\left[\exp\left(i\, \left[f * d\hat{X}\right] + i \,\left[ h * d\hat{M}\right]\right) \right]\,,
    $$
    so that finally
    $$
    \mathbb{E}\left[\exp\left(i\, \left[f * d\hat{X}\right] + i \, \left[h * d\hat{M}\right]\right)\right] = \exp\left(g_0\,\int_0^T \, \Psi(T-t)\, dt\right)\,.
    $$
    From Remark \ref{remark:Psi_is_IG}, this means that for any $f\,,h$ of class $\mathcal{C}^1\,$,
    \begin{equation}
    \label{eq:convolution_charac_X_M}
    \mathbb{E}\left[\exp\left(i\, \left[f * d\hat{X}\right] + i \, \left[h * d\hat{M}\right]\right)\right] = \mathbb{E}\left[\exp\left(i\, \left[(f + (1 + \lambda)\,h) * dY\right] - i \, \left[h * dG_0\right]\right) \right]
    \end{equation}
    where $Y$ is the IG process defined at the beginning of this proof.
    Let $r \in \N\,$, $(u_1\,,\ldots\,,u_r)$ and $(v_1\,,\ldots\,,v_r)$ in $\R^r$ and $0 = t_0 < t_1 < t_2 < \ldots < t_r =  T\,$. We want to prove that 
    \begin{equation}
    \label{eq:finite_dimensional_laws}
    \mathbb{E}\left[\exp\left(i\, \sum_{j=1}^r \, \left[u_j\,\hat{X}_{t_j} + v_j \, \hat{M}_{t_j}\right]\right) \right] = \mathbb{E}\left[\exp\left(i\,\sum_{j = 1}^r \, \left[\left(u_j + (1+\lambda)\, v_j \right)\, Y_{t_j} - v_j \,G_0(t_j)\right]\right) \right]\,.
    \end{equation}
    Introducing 
    $$
    \tilde{u}_i := \sum_{j = i}^r\,u_j\,, \quad 1 \leq i \leq r\,,
    $$
    we have
    $$
    \sum_{i = 1}^r\,u_i\,\hat{X}_{t_i} = \sum_{i = 1}^r\, \tilde{u}_i \, \left(\hat{X}_{t_i} - \hat{X}_{t_{i - 1}}\right)\,,
    $$
    since $\hat X_0 = 0\,$.
    For $k > \left(\min\limits_{1 \leq i \leq r}\, t_i - t_{i-1} \right)^{-1}$, we can define
    $$
    f_k(t) := \begin{cases}
        \tilde{u}_i\,, \quad &\text{if $T - t_i \leq t < T - t_{i-1} - \frac{1}{k}$ for $2 \leq i \leq r\,$,}\\
        \tilde{u}_{i} + (\tilde{u}_{i-1} -\tilde{u}_{i})\,\phi\left(1 - k\,(T - t_{i-1} - t) \right)\,, \quad &\text{if $T - t_{i -1} - \frac{1}{k} \leq t < T - t_{i-1}$ for $2 \leq i \leq r\,$,}\\
        \tilde{u}_1\,, \quad &\text{if $T - t_1 \leq t \leq T$}\,,
    \end{cases}
    $$
    where $\phi : [0, 1] \to [0,1]$ is any nondecreasing function of class $\mathcal{C}^1$ such that $\phi(0) = \phi'(0) = \phi'(1) = 0$ and $\phi(1) = 1\,$. For example, one can use $\phi : x \mapsto 3x^2 - 2x^3\,$. We can easily check that $f_k$ is of class $\mathcal{C}^1$ on $[0\,,T]$ for any $k\,$, along with 
    $$
    \sup_k\, \left\| f_k\right\|_{\infty} \leq \max_{1 \leq i \leq r}\, \left|\tilde{u}_i\right|\,,
    $$
    and the fact that $\left(f_k\right)_{k}$ converges pointwise to 
    $$
    t \mapsto  \tilde{u}_1 \, \1_{[T-t_1,T]}(t) + \sum_{i = 2}^r\, \tilde{u}_i\, \1_{[T - t_i,T - t_{i-1})}(t)  \,.
    $$
    Using the dominated convergence theorem, we obtain almost surely
    $$
    \int_0^T\,f_k(T-t)\, d\hat{X}_t \overset{k \to \infty}{\longrightarrow} \sum_{i = 1}^r \, \tilde{u}_i\, \left(\hat{X}_{t_i} - \hat{X}_{t_{i - 1}}\right) = \sum_{i = 1}^r\, u_i\, \hat{X}_{t_i}\,.
    $$
    Replacing $(u_1\,,\ldots\,,u_r)$ with $(v_1\,, \ldots \,,v_r)\,$, we can construct a sequence of functions $\left(h_k\right)_k$ of class $\mathcal{C}^1$ such that
    $$
    \int_0^T\,h_k(T-t)\,d\hat{M}_t \overset{k \to \infty}{\longrightarrow} \sum_{i = 1}^r \, v_i\, \hat{M}_{t_i}\,, 
    $$
    in probability, using the dominated convergence theorem for Itô integrals, see \citet[Theorem~I.4.31]{jacod2013limit}. Finally, we have 
    $$
    \mathbb{E}\left[\exp\left(i\, \left[f_k * d\hat{X}\right] + i \, \left[h_k * d\hat{M} \right] \right) \right] \overset{k \to \infty}{\longrightarrow} \mathbb{E}\left[\exp\left(i\, \sum_{j = 1}^r\, u_j\, \hat{X}_{t_j} + i \, \sum_{j = 1}^r\, v_j\, \hat{M}_{t_j} \right)\right]
    $$
    since the sequence is bounded by $1$ and must therefore converge to its unique accumulation point, which is given by the convergence in probability of the integrals.
    Since the inverse Gaussian process $Y$ and the function $G_0$ are nondecreasing, the same reasoning gives 
    $$
    \mathbb{E}\left[\exp\left(i\, \left[(f_k + (1 + \lambda)\,h_k) * dY\right] - i\,\left[h_k * dG_0\right] \right)\right] \overset{k \to \infty}{\longrightarrow} \mathbb{E}\left[\exp\left(i\, \sum_{j = 1}^r\,\left(u_j + (1+\lambda)\,v_j\right)\, Y_{t_j}  - v_j\,G_0(t_j)\right) \right]\,.
    $$
    Combining these two results with \eqref{eq:convolution_charac_X_M} shows \eqref{eq:finite_dimensional_laws}. Thus $\left(\hat{X}\,, \hat{M}\right) \sim \left(Y\,, (1 + \lambda)\, Y - G_0\right)\,$, which concludes the proof.

\subsection{Proof of Lemma \ref{lemma:estimates}}
\label{subsec:proof_estimates}
We first prove the martingality of $M^n$ for all $n \geq 0\,$, from which we easily obtain the estimates on $\left(X^n\right)_{n \geq 0}\,$. Beforehand, we need to introduce the notion of \textit{resolvent} of a convolution kernel.

\paragraph{Resolvent of the fractional kernel.}
For any $K \in L^1_{loc}(\R^+)\,$, there exists a unique $R \in L^1_{loc}(\R^+)\,$, called its \textit{resolvent} or \textit{resolvent of the second kind} \citep*[Theorem~2.3.1]{gripenberg1990volterra}\footnote{The book defines the resolvent as the solution to $K * R = K - R$ instead of (\ref{eq:resolvent_def}), which corresponds to $-1$ times the resolvent of $-K\,$. Therefore the existence and uniqueness still hold.}, satisfying 
\begin{equation}
    \label{eq:resolvent_def}
\left(K * R\right)(t) = \left(R * K\right)(t) = R(t) - K(t) \,, \quad t \geq 0\,,
\end{equation}
where we use the convolution notation introduced in \eqref{eq:def_convolution}.
For $n \geq 0\,$,
and $\alpha \geq 0\,$, the resolvent of $\alpha\, K^n$ (see \eqref{eq:Kn}) is given by
\begin{equation}
    \label{eq:fractional_resolvent}
R_{\alpha}^n(t) := \alpha \, h_n\,\Gamma\left(h_n\right) t^{h_n - 1}\, E_{h_n, h_n}\left(\alpha \, h_n\,\Gamma\left(h_n\right)\,t^{h_n} \right)\,, \quad t > 0\,,
\end{equation}
where $h_n := H^n + 1/2$ and $E_{a,b}(z) := \sum_{k = 0}^{+ \infty}\, \frac{z^k}{\Gamma(ak + b)}$ is the Mittag-Leffler function (see for instance \citet*[4.4.5]{gorenflo2020mittag}, using $E_{\beta, \beta}(z) = \Gamma(\beta)^{-1} +z\,E_{\beta, 2\beta}(z)$).

\paragraph{Martingality of $M^n$.} 
Let $n \geq 0\,$, for any $k \geq 0$ we define the stopping time $\tau_k := \inf \{s \geq 0 \,:\, X^n_s \, \vee \left| M^n_s\right| \geq k \}\, \wedge \, T\,$, so that $\left(\tau_k\right)_{k \geq 0}$ almost surely converges to $T$ as $k$ goes to infinity. By continuity, the stopped process $\left(X^n\right)^{\tau_k} := \left(X^n_{t\,\wedge\,\tau_k}\right)_{ t \leq T}$ is bounded by $k$ and $\left(M^n \right)^{\tau_k} := \left(M^n_{t \,\wedge\,\tau_k}\right)_{t \leq T}$ is a true square-integrable martingale with quadratic variation $\nu^2 \left(X^n\right)^{\tau_k}\,$. By nonnegativity of $\lambda$ and $X^n\,$, we have
\begin{equation}
\label{eq:X_first_inequality}
X^n_t \leq G_0^n(t) + \int_0^t \, M^n_{t-s}\, K^n(s)\, ds\,, \quad t \leq T\,,
\end{equation}
from \eqref{eq:definition_sequence},
which implies
\begin{align*}
    X^n_{t \, \wedge \, \tau_k} &\leq G_0^n(t \, \wedge \, \tau_k) + \int_0^{t \, \wedge \, \tau_k}\, \left|M^n_{t \, \wedge \, \tau_k - s}\right|\, K^n(s) \, ds \\
    &\leq G_0^n(T) + \int_0^t \, \left|M^n_{t \, \wedge \tau_k - s}\right| \, \1_{s \leq \tau_k}\, K^n(s)\, ds \,, \quad t \leq T\,, \quad k \geq 0\,.
\end{align*}
We then remark that for any $k \geq 0$ and $s \leq t \leq T\,$, $\left| M^n_{t \, \wedge \, \tau_k - s} \right|\, \1_{s \leq \tau_k} \leq \sup\limits_{0 \leq u \leq t - s}\, \left| \left( M^n \right)^{\tau_k}_{u} \right|$ almost surely, so that by the BDG inequality \citep[Theorem~IV.54]{protter2005stochastic},
$$
\mathbb{E}\left[\left|M^n_{t \,\wedge \, \tau_k -s} \right|^2 \, \1_{s \leq \tau_k}\right] \leq \mathbb{E}\left[\left(\sup_{0 \leq u \leq t-s} \, \left|\left(M^n\right)^{\tau_k}_u \right|\right)^2\right] \leq 4\,\nu^2\,\mathbb{E}\left[\left(X^n\right)^{\tau_k}_{t-s}\right]\,, \quad s \leq t \leq T\,, \quad k \geq 0\,,
$$
and therefore, using $|x| \leq 1 + x^2\,$,
\begin{align*}
    \mathbb{E}\left[X^n_{t \, \wedge \, \tau_k} \right] &\leq G_0^n(T) + \int_0^t \, \left(1 + 4\,\nu^2\,\mathbb{E}\left[X^n_{(t-s) \, \wedge \, \tau_k}  \right] \right)\, K^n(s)\, ds \\
    &= G_0^n(T) + \int_0^T\, K^n(s)\, ds + 4\, \nu^2\, \int_0^t \,\mathbb{E}\left[X^n_{s \, \wedge \, \tau_k}  \right]\, K^n(t-s)\, ds\,, \quad t \leq T\,, \quad k \geq 0\,. 
\end{align*}
Thus, knowing that $R^n_{4\,\nu^2}$ is nonnegative from its definition \eqref{eq:fractional_resolvent}, and using the Grönwall lemma for convolution inequalities from \citet*[Theorem~9.8.2]{gripenberg1990volterra}, we get
$$
\mathbb{E}\left[X^n_{t \, \wedge \, \tau_k} \right] \leq \left(G_0^n(T) + \int_0^T\, K^n(s)\, ds  \right) \, \left(1 + \int_0^T\, R^n_{4 \, \nu^2} (s) \, ds \right) \,, \quad t \leq T\,, \quad k \geq 0\,.
$$
Finally, applying Fatou's lemma as $k \to \infty\,$, we have $\mathbb{E}\left[X^n_t \right] < + \infty$ for any $0 \leq t \leq T\,$. According to \citet[Theorem~II.27, Corollary~3]{protter2005stochastic}, a local martingale with integrable quadratic variation is a true square-integrable martingale. Therefore, $M^n$ is a true martingale on $[0\,, T]$ with $\mathbb{E}\left[\left(M^n_t\right)^2 \right] = \nu^2\, \mathbb{E}\left[X^n_t\right] < + \infty\,$ for all $t \leq T\,$.

\paragraph{Estimates on $X^n$.}
We can now go back to \eqref{eq:X_first_inequality} and take the expectation to obtain
$$
\mathbb{E}\left[X^n_t\right] \leq G_0^n(t)\,, \quad t \leq T\,.
$$
One can also write
\begin{align*}
    X^n_t - X^n_s &= G_0^n(t) - G_0^n(s) + \int_0^s \left[-\lambda \,\left(X^n_{t - u} - X^n_{s - u}\right) + M^n_{t-u} - M^n_{s - u} \right]\, K^n(u)\, du \\
    &\quad + \int_s^t \left(- \lambda \, X^n_{t-u} + M^n_{t-u}\right)\, K^n(u) \, du\\
    &\leq G_0^n(t) - G_0^n(s) + \int_0^t \, \left(M^n_{t-u} - M^n_{s-u} \right)\, K^n(u)\, du + \int_s^t\, M^n_{t-u}\, K^n(u)\, du\,, \quad s \leq t \leq T\,,
\end{align*}
where we used $X^n_{t-u} - X^n_{s-u} \geq 0$ since $X^n$ is nondecreasing, which gives 
$$
\mathbb{E}\left[X^n_t - X^n_s \right] \leq G_0^n(t) - G_0^n(s)\,, \quad s \leq t \leq T\,.
$$

\subsection{Proof of Lemma \ref{lemma:tightness}}
\label{subsec:proof_tightness}
\paragraph{$M_1$-tightness of $\left(X^n\right)_{n\geq 0}\,$.}
Since for any $n \geq 0\,$, $X^n$ is a nondecreasing, nonnegative process, starting from $0\,$, the conditions from Proposition \ref{proposition:M_1_tightness} for tightness in the $M_1$ topology are reduced to (see Remark \ref{remark:M1_non_decreasing_processes}):
    \begin{itemize}
    \item[$\bullet$] $\lim\limits_{R \to + \infty} \, \sup\limits_n \, \mathbb{P}\left( X^n_T > R\right) = 0$\,,
    \item[$\bullet$] for all $\eta > 0 \,,\,\, \lim\limits_{\delta \to 0^+}\, \sup\limits_n \, \mathbb{P}\left( X^n_{\delta} \, \vee \, \left(X^n_{T} - X^n_{T - \delta} \right) \geq \eta\right) = 0$\,.
\end{itemize}
Using Lemma \ref{lemma:estimates} and Markov's inequality, we first have 
$$
\sup_n \, \mathbb{P}\left(X_T^n > R \right) \leq \frac{\sup_n\,G_0^n(T)}{R} \,\overset{R \to \infty}{\longrightarrow}\, 0\,.
$$
Then, using the definition of $\left(G_0^n\right)_{n \geq 0}$ in \eqref{eq:G_0_n_def}, we have for $\delta < \min(1\,,T)\,$,
\begin{align*}
\sup_n \, \mathbb{E}\left[X_{\delta}^n \right] &\leq \sup_n \, G_0^n(\delta) \\
&\leq \delta \, \left(V_0 + \sup_n\,\frac{\lambda \, \theta}{H^n + 3/2}\, \delta^{H^n + 1/2}\right)\\
&\leq \delta \,\left(V_0 + \sup_n \, \frac{\lambda\,\theta}{H^n + 3/2} \right) \\
&\leq \delta \, \left(V_0 + \lambda \, \theta \right)\,,
\end{align*}
and similarly,
\begin{align*}
\sup_n \, \mathbb{E}\left[X_T^n - X_{T - \delta}^n \right] &\leq \sup_n \,\left(G_0^n(T) - G_0^n(T - \delta)\right) \\
&\leq V_0 \, \delta + \sup_n \, \frac{\lambda \, \theta}{H^n + 3/2}\, \left(T^{H^n + 3/2} - \left(T - \delta\right)^{H^n + 3/2}\right) \\
&\leq \delta \, \left(V_0 + \lambda \, \theta \, \sup_n \, T^{H^n + 1/2}\right)\,,
\end{align*}
using the two estimates on $\left(X^n\right)_{n \geq 0}$ given in Lemma \ref{lemma:estimates}. Therefore,
$$
\lim_{\delta \to 0^+}\, \sup_n \, \mathbb{E}\left[X^n_{\delta}\right] = \lim_{\delta \to 0^+}\, \sup_n \, \mathbb{E}\left[X^n_T - X^n_{T - \delta}\right] = 0\,,
$$
which concludes the proof, by Markov's inequality. The sequence $\left(X^n\right)_{n \geq 0}$ is tight for the $M_1$ topology.

\paragraph{$S$-tightness of $\left(M^n\right)_{n\geq 0}\,$.} The condition from Proposition \ref{proposition:S_tightness_supermartingale} for $S$-tightness of a sequence of supermartingales is easily verified for $\left(M^n\right)_{n \geq 0}\,$. To see this, we apply Lemma \ref{lemma:estimates},
$$
\mathbb{E}\left[\left|M^n_t \right| \right] \leq 1 + \nu^2\, \mathbb{E}\left[X^n_t \right] \leq 1 + \nu^2\, G_0^n(t)\,, \quad t \leq T\,, \quad n \geq 0\,.
$$
Since for all $n \geq 0$ the function $G_0^n$ is nondecreasing,
$$
\sup_{0\leq t \leq T}\, \mathbb{E}\left[\left|M^n_t\right|\right] \leq 1 + \nu^2\,G_0^n(T)\,, \quad n \geq 0\,,
$$
and finally, using $G_0^n(T) \overset{n \to \infty}{\longrightarrow} g_0\,T$ (see \eqref{eq:G_0_n_def}), we get
$$
\sup_n \, \sup_{0 \leq t \leq T}\, \mathbb{E}\left[\left|M^n_t\right|\right] \leq 1 + \nu^2\, \sup_n\, G_0^n(T) < + \infty\,.
$$
The sequence $\left(M^n\right)_{n \geq 0}$ is $S$-tight.

\paragraph{$M_1 \otimes S$-tightness of $\left(X^n\,, M^n\right)_{n\geq 0}\,$.}
By Tychonoff's theorem, the Cartesian product of an $M_1$-compact set with an $S$-compact set is $M_1 \otimes S$-compact. Therefore, tightness of $(X^n)_{n \geq 0}$ for $M_1$ along with tightness of $(M^n)_{n \geq 0}$ for $S$ gives the $M_1 \otimes S$-tightness of $(X^n\,, M^n)_{n\geq 0}\,$.

\subsection{Proof of Lemma \ref{lemma:characteristic_convergence}}
\label{subsec:proof_characteristic_convergence}
The continuity of $\Psi$ is direct from the form given in Remark \ref{remark:Psi_is_IG} since $f$ and $h$ are continuous and the real part of the complex number under the square-root is positive.

We aim to prove the convergence 
\begin{equation}
\label{eq:goal_cv_characteristic_proof}
\mathcal{I}^n := \int_0^T \, F\left(T-t\,, \Psi^n(T-t) \right)\,dG_0^n(t) \overset{n \to \infty}{\longrightarrow} \mathcal{I} := \int_0^T \, F\left(T-t\,, \Psi(T-t)\right)\,dG_0(t)\,.
\end{equation}
First, we note that it is enough to find a constant $C$ such that
\begin{equation}
    \label{eq:F_Psi_n_bound}
\left|F\left(t\,, \Psi^n(t)\right)\right| \leq C\,, \quad t \leq T \,, \quad n \geq 0\,,
\end{equation}
and to show that
\begin{equation}
    \label{eq:delta_F_to_0_L_1}
    \int_0^T\, \left|F\left(t\,, \Psi^n(t)\right) - F\left(t\,, \Psi(t)\right) \right|\,dt \overset{n \to \infty}{\longrightarrow} 0\,.
\end{equation}
If that is the case, recalling that $dG_0^n(t) = \left(V_0 + \lambda\, \theta \, t^{H^n + 1/2}\right)\,dt$ and $dG_0(t) = g_0\,dt$ (see \eqref{eq:G_0_n_def} and \eqref{eq:G_0_def}),
\begin{align*}
    \left| \mathcal{I}^n - \mathcal{I}\right| &\leq \int_0^T \, \left|F\left(T-t\,, \Psi^n(T-t)\right) - F\left(T-t\,, \Psi(T-t)\right) \right| \,dG_0(t) \\
    &\quad + \int_0^T \, \left| F\left(T-t\,, \Psi^n(T-t)\right)\right|\, \left|dG_0(t) - dG_0^n(t)\right| \\
    &\leq g_0 \, \int_0^T\,\left|F\left(t\,, \Psi^n(t)\right) - F\left(t\,, \Psi(t)\right) \right| \,dt + C\, \lambda \, \theta \, \int_0^T \, \left| 1 - t^{H^n + 1/2} \right|\,dt\,,\quad n \geq 0\,,
\end{align*}
which vanishes as $n$ goes to infinity thanks to \eqref{eq:delta_F_to_0_L_1} and to the dominated convergence theorem, yielding \eqref{eq:goal_cv_characteristic_proof}.

It remains to prove \eqref{eq:delta_F_to_0_L_1} and the existence of  a constant $C$ satisfying \eqref{eq:F_Psi_n_bound}. We rely on the comparison result for linear Volterra equations given in \citet[Theorem~C.3]{abi2019multifactor}\footnote{The theorem was proved for $K \in L^2_{loc}$ but the proof only uses $K \in L^1_{loc}$ (as well as the existence of solutions to the linear Volterra equations and some properties on the resolvent of the first kind, which are given by \citet[Example~2.4, Theorem~2.5]{abi2021weak} since for any $n \geq 0$ the kernel $K^n$ is completely monotone). Thus, we can use this result in our context.}. Introducing $a := \nu^2\,/\,2$ as well as $\rho := - \lambda + i\nu^2\, h$ and $\xi := i\,f - \nu^2\,h^2 \,/\,2\,$, we have by definition
$$
F\left(t,u\right) = a\,u^2 + \rho(t)\,u + \xi(t)\,, \quad t \leq T \,, \quad u \in \mathbb{C}\,.
$$
We recall that $\left(\Psi^n\right)_{n \geq 0}$ satisfies on $[0\,,T]\,$:
$$
\Psi^n = K^n * \left(a\,\left(\Psi^n\right)^2 + \rho\,\Psi^n + \xi\right) = K^n * \left(z^n \, \Psi^n + \xi\right)\,, \quad n \geq 0\,,
$$
where $z^n := a \, \Psi^n + \rho\,$. Thus, applying \citet[Theorem~C.3]{abi2019multifactor}, for all $n \geq 0$ we have $\left|\Psi^n\right| \leq \chi^n$ on $[0\,,T]\,$, where $\chi^n \geq 0$ is the solution to 
$$
\chi^n(t) = \left(K^n * \left[\Re\left(z^n\right)\chi^n + \left|\xi\right|\,\right]\right) (t)\,, \quad t \leq T\,.
$$
Since $\Re\left(z^n\right) = \nu^2 \, \Re\left(\Psi^n\right)\,/\,2 - \lambda \leq 0$ (recall that $\Psi^n$ is known to have nonpositive real part for all $n \geq 0$ from \citet[Theorem~2.5]{abi2021weak}), 
$$
\chi^n(t) \leq \left|\xi(t)\right|\,\int_0^T\, K^n(s)\,ds = \left|\xi(t)\right| \,T^{H^n + 1/2}\,, \quad t \leq T \,, \quad n \geq 0\,,
$$
which can be bounded uniformly in $t$ and in $n$ since $\xi$ is continuous and $H^n \to -1/2\,$. Thus, 
$$
\sup_n \, \sup_{0 \leq t \leq T}\, \left| \Psi^n(t) \right| < + \infty\,,
$$
and from the continuity of $\rho$ and $\xi\,$, we can find a constant $C \geq 0$ satisfying \eqref{eq:F_Psi_n_bound}. We now introduce 
$$
\Delta^n(t) := \int_0^t\, \Psi(s)\, K^n(t-s)\,ds - \Psi(t)\,, \quad t \leq T\,, \quad n \geq 0\,.
$$
Recalling that $\Psi(t) = F\left(t\,, \Psi(t)\right)$ on $[0\,,T]\,$, we have 
\begin{align*}
    \Psi^n - \Psi &= K^n * \left(a\,\left(\Psi^n\right)^2 + \rho \, \Psi^n + \xi\right) - K^n * \Psi + \Delta^n \\
    &= K^n * \left(a\, \left[\left(\Psi^n\right)^2 - \Psi^2\right] + \rho \, \left[\Psi^n - \Psi\right]\right) + \Delta^n\\
    &= K^n * \left(\tilde{z}^n\, \left(\Psi^n - \Psi\right)\right) + \Delta^n \,, \quad n \geq 0\,,
\end{align*}
where $\tilde{z}^n := a\, \left(\Psi^n + \Psi\right) + \rho\,$.
This can be rewritten as
$$
\Psi^n - \Psi - \Delta^n = K^n * \left(\tilde{z}^n \, \left(\Psi^n - \Psi - \Delta^n\right) + \tilde{z}^n \, \Delta^n\right)\,,\quad n \geq 0\,.
$$
Another application of \citet[Theorem~C.3]{abi2019multifactor} yields 
$$
\left|\Psi^n(t) - \Psi(t) - \Delta^n(t)\right| \leq \tilde{\chi}^n(t) \,, \quad t \leq T\,, \quad n \geq 0\,,
$$
where for each $n \geq 0\,$, the function $\tilde{\chi}^n \geq 0$ is the solution to
$$
\tilde{\chi}^n(t) = \left(K^n * \left[\Re\left(\tilde{z}^n\right) \tilde{\chi}^n + \left|\tilde{z}^n\, \Delta^n\right|\,\right]\right)(t) \,, \quad t \leq T\,.
$$
Since $\Re\left(\tilde{z}^n\right) = \nu^2 \,\Re\left(\Psi^n + \Psi\right) \,/\,2 - \lambda \leq 0$ for any $n\geq 0$ because $\Re\left(\Psi^n\right)$ and $\Re\left(\Psi\right)$ are nonpositive, we get
$$
\left|\Psi^n(t) - \Psi(t) - \Delta^n(t)\right| \leq \left( K^n * \left|\tilde{z}^n \, \Delta^n\right|\right)(t) \,, \quad t \leq T \,, \quad n \geq 0\,,
$$
which gives
$$
\left|\Psi^n(t) - \Psi(t)\right| \leq \left(K^n * \left|\tilde{z}^n \, \Delta^n\right|\right)(t) + \left|\Delta^n(t)\right|\,, \quad t \leq T\,, \quad n \geq 0\,.
$$
Finally,
\begin{align*}
    \left|F\left(t \,, \Psi^n(t)\right) - F\left(t\,, \Psi(t)\right) \right| &= \left| \tilde{z}^n(t) \right| \, \left| \Psi^n(t) - \Psi(t)\right| \\
    &\leq \left| \tilde{z}^n(t)\right| \, \left(K^n * \left| \tilde{z}^n \, \Delta^n\right|\right)\,(t) + \left| \tilde{z}^n(t) \, \Delta^n(t)\right| \,,\quad t \leq T\,, \quad n \geq 0\,.
\end{align*}
We have already proved that $\left(\Psi^n\right)_{n \geq 0}$ is uniformly bounded in $t$ and in $n\,$, so that by continuity of $\Psi$ and $\rho\,$, we can find a constant $\tilde{C} \geq 0$ such that $\left|\tilde{z}^n(t)\right| \leq \tilde{C}$ for all $n \geq 0$ and $t \leq T\,$. Thus, setting 
$$
\varepsilon^n := \tilde{C}^2\, K^n * \left|\Delta^n\right| + \tilde{C} \, \left|\Delta^n\right| \,, \quad n \geq 0\,,
$$
we have 
\begin{equation}
\label{eq:delta_F_bound_h}
\left|F\left(t\,, \Psi^n(t)\right) - F\left(t\,, \Psi(t)\right)\right| \leq \varepsilon^n(t) \,, \quad t \leq T\,, \quad n \geq 0\,.
\end{equation}
For any $n\geq 0\,$, the function $\varepsilon^n$ is continuous and therefore integrable on $[0\,,T]\,$. Applying Fubini's theorem,
\begin{align}
\label{eq:integral_h}
    \int_0^T\, \varepsilon^n(t)\,dt &= \tilde{C}^2\,\int_0^T \, \left(K^n * \left|\Delta^n\right|\right)(t)\,dt + \tilde{C}\, \int_0^T\, \left|\Delta^n(t)\right|\,dt \nonumber \\
    &= \tilde{C}^2\, \int_0^T\, \left|\Delta^n(s)\right|\, \int_0^{T-s}\,K^n(t)\,dt\,ds + \tilde{C}\, \int_0^T\, \left|\Delta^n(t)\right|\,dt \nonumber \\
    &\leq \tilde{C}\, \left(\tilde{C}\,T^{H^n + 1/2} + 1\right)\, \int_0^T\, \left|\Delta^n(t)\right|\,dt\,, \quad n \geq 0\,.
\end{align}
The weak convergence of the measures $\left(K^n(t) \, dt\right)_{n \geq 0}$ to $\delta_0$ in the strengthened form given by Lemma \ref{lemma:weak_cv_to_dirac}, along with the continuity of $\Psi$, give
$$
\Delta^n(t) \overset{n \to \infty}{\longrightarrow} 0 \,,\quad 0 < t \leq T\,,
$$
and one easily has
$$
\sup_n \, \sup_{0 \leq t \leq T}\, \left|\Delta^n(t)\right| \leq \left(1 + \sup_n \, T^{H^n + 1/2}\right)\, \sup_{0 \leq t \leq T} \, \left|\Psi(t)\right| < + \infty\,,
$$
which proves that $\int_0^T \, \left|\Delta^n(t)\right|\, dt \overset{n \to \infty}{\longrightarrow} 0$ thanks to the dominated convergence theorem. Combining this with \eqref{eq:integral_h} and \eqref{eq:delta_F_bound_h} proves \eqref{eq:delta_F_to_0_L_1}, which concludes the proof.

\newpage
\begin{appendix}
\section{The $M_1$ topology}\label{A:M1_topology}

When wanting to prove weak convergence of càdlàg stochastic processes, one needs to endow $\mathbb{D}_T\,$, the space of càdlàg functions on $[0\,,T]\,$, with a topology making it a Polish space (i.e., separable completely metrizable), in order to use the Prokhorov theorem \citep[Theorem~11.6.1]{whitt2002stochastic}. In his seminal paper, \citet{skorokhod1956limit} introduced several topologies satisfying this property. The most commonly used is the $J_1$ topology obtained by deforming the time over the uniform topology. More precisely, it is generated by the metric
$$
d_{J_1}(f\,, g) := \inf_{\lambda \in \Lambda}\, \left\{\| f \circ \lambda - g \|_{\infty} + \| \lambda \|_{\circ} \right\}\,, \quad \left(f\,,g\right) \in \mathbb{D}_T^2\,,
$$
where $\Lambda$ is the set of increasing bijections on $[0\,,T]\,$, and 
$$
\|\lambda\|_{\circ} := \sup_{0 \leq t \neq s \leq T}\, \left|\log \left(\frac{\lambda(t) - \lambda(s)}{t -s} \right) \right|\,, \quad \lambda \in \Lambda\,,
$$
penalizes the fact that $\lambda$ is ``too far'' from $x \mapsto x\,$. A visualization of the $\varepsilon$-tubes (the ball of radius $\varepsilon\,$, centered in a certain function, for a metrizable topology) obtained in this topology, compared to those of the uniform topology, is given in Figure \ref{fig:eps_tubes} (this plot is inspired by \citet{kern2022skorokhod}). We can observe that the sequence of càdlàg functions $\left(\1_{[1/2 - 1/n\,, 1]}\right)_n$ converges to $\1_{[1/2\,, 1]}$ for the $J_1$ topology on $\mathbb{D}_1\,$, whereas this is not the case for the uniform topology.

\begin{figure}[h]
    \centering
\begin{tikzpicture}
    \draw[->] (-0.5,0) -- (4.2,0) node[right] {};
    \draw[->] (0,-0.5) -- (0,4.2) node[above] {};

    \draw[thick] (0, 0.1) -- (0, -0.1) node[below left] {0};
    \foreach \x in {0.5, 1} 
        \draw[thick] (\x * 4, 0.1) -- (\x * 4, -0.1) node[below] {\x};
    
    \draw[thick] (0.1, 4) -- (-0.1, 4) node[left] {1};
    
    
    \draw[very thick, blue] (0, 0) -- (2, 0);
    \draw[very thick, blue] (2, 4) -- (4, 4);
    \draw[dashed, red] (0, 0.2) -- (2, 0.2);
    \draw[dashed, red] (0, -0.2) -- (2, -0.2);
    \draw[dashed, red] (2, -0.2) -- (2, 0.2);
    \draw[dashed, red] (2, 4.2) -- (4, 4.2);
    \draw[dashed, red] (2, 3.8) -- (4, 3.8);
    \draw[dashed, red] (2, 3.8) -- (2, 4.2);

    \draw[fill = white] (2, 0) circle (3pt);
    \draw[fill = blue] (2, 4) circle (3pt);

    \draw[->] (4.5, 0) -- (9.2, 0) node[right] {};
    \draw[->] (5, -0.5) -- (5, 4.2) node[above] {};

    \draw[thick] (5, 0.1) -- (5, -0.1) node[below left] {0};
    \draw[thick] (7, 0.1) -- (7, -0.1) node[below] {0.5};
    \draw[thick] (9, 0.1) -- (9, -0.1) node[below] {1};

    \draw[thick] (5.1, 4) -- (4.9, 4) node[left] {1};

    \draw[very thick, blue] (5,0) -- (7, 0);
    \draw[very thick, blue] (7,4) -- (9, 4);
    \draw[dashed, red] (5, 0.2) -- (7.2, 0.2);
    \draw[dashed, red] (5, -0.2) -- (7.2, -0.2);
    \draw[dashed, red] (7.2, -0.2) -- (7.2, 0.2);
    \draw[dashed, red] (6.8, 4.2) -- (9, 4.2);
    \draw[dashed, red] (6.8, 3.8) -- (9, 3.8);
    \draw[dashed, red] (6.8, 3.8) -- (6.8, 4.2);

    \draw[fill = white] (7, 0) circle (3pt);
    \draw[fill = blue] (7, 4) circle (3pt);

    \draw[->] (9.5, 0) -- (14.2, 0) node[right] {};
    \draw[->] (10, -0.5) -- (10, 4.2) node[above] {};

    \draw[thick] (10, 0.1) -- (10, -0.1) node[below left] {0};
    \draw[thick] (12, 0.1) -- (12, -0.1) node[below] {0.5};
    \draw[thick] (14, 0.1) -- (14, -0.1) node[below] {1};
    \draw[thick] (10.1, 4) -- (9.9, 4) node[left] {1};

    \draw[very thick, blue] (10, 0) -- (12, 0);
    \draw[very thick, blue] (12, 4) -- (14, 4);
    \draw[dashed, red] (10, 0.2) -- (11.8, 0.2);
    \draw[dashed, red] (10, -0.2) -- (12.2, -0.2);
    \draw[dashed, red] (11.8, 0.2) -- (11.8, 4.2);
    \draw[dashed, red] (12.2, -0.2) -- (12.2, 3.8);
    \draw[dashed, red] (11.8, 4.2) -- (14, 4.2);
    \draw[dashed, red] (12.2, 3.8) -- (14, 3.8);

    \draw[fill = white] (12, 0) circle (3pt);
    \draw[fill = blue] (12, 4) circle (3pt);
\end{tikzpicture}
\caption{Example of $\varepsilon$-tube (red dashed lines) for the function $\1_{[1/2\,, 1]}$ (blue plain lines) for different topologies. \textbf{Left:} Uniform topology. \textbf{Middle:} $J_1$ topology. \textbf{Right:} $M_1$ topology.}
\label{fig:eps_tubes}
\end{figure}
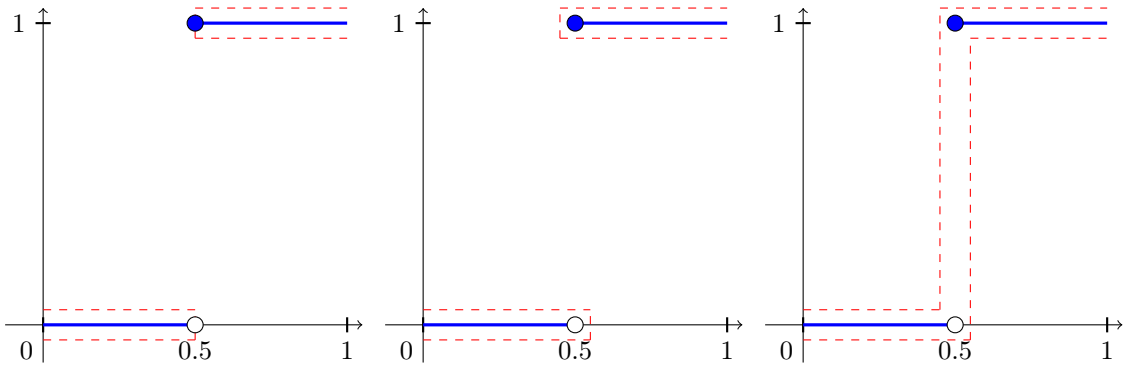

Unfortunately, even though $J_1$ is weaker than the uniform topology, the $J_1$-limit of a sequence of continuous functions must still be continuous \citep[Section~12]{billingsley2013convergence}. For instance, the sequence of piecewise affine functions $\left(f_n\right)_{n \geq 2}$ defined by 
$$
f_n(x) := \begin{cases}
    0 \,, \quad &\text{if $x \in [0\,, 1/2 - 1/n)\,$} \\
    n \, (x - 1 /2) + 1 \,, \quad &\text{if $x \in [1/2 - 1 / n\,, 1/2)$\,} \\
    1 \,, \quad &\text{if $x \in [1/2\,, 1]$\,}
\end{cases} \,, \quad n \geq 2\,,
$$
does not converge to $\1_{[1/2\,, 1]}$ in $\left(\mathbb{D}_1\,, J_1\right)\,$.

To solve this issue, one can consider the weaker $M_1$ topology (see Figure \ref{fig:eps_tubes} to visualize its $\varepsilon$-tubes) that allows the convergence $f_n \overset{n \to \infty}{\longrightarrow} \1_{[1/2\,, 1]}$ and, more generally, allows continuous functions to converge to a discontinuous limit. The idea is to allow for deformations of time and space. It is based on the \textit{thin graph} of a càdlàg function, defined by 
$$
\Gamma_x := \left\{(z\,, t) \in \R \times [0\,,T] \,, \, z \in [x(t-)\,,x(t)] \right\}\,, \quad x \in \mathbb{D}_T\,.
$$
We introduce an order on $\Gamma_x\,$, given by 
$$
(z_1\,, t_1) \leq (z_2\,,t_2) \Longleftrightarrow \left(\text{$t_1 < t_2$ \quad or \quad $t_1 = t_2$ and $|z_1 - x(t-)| \leq |z_2 - x(t-)|$}\right)\,,
$$
which allows one to consider nondecreasing functions with values in $\Gamma_x\,$. Then, let 
$$
\Pi_x := \left\{\text{$(\rho\,, \lambda)$ continuous, nondecreasing, mapping $[0\,, 1]$ onto $\Gamma_x$}  \right\}\,, \quad x \in \mathbb{D}_T\,,
$$ be the set of parametrizations of $\Gamma_x\,$. The metric generating $M_1$ is 
$$
d_{M_1}(f,g) = \inf_{\substack{(\rho_f\,, \lambda_f) \in \Pi_f \\ (\rho_g\,, \lambda_g) \in \Pi_g}} \, \left\{\|\rho_f - \rho_g\|_{\infty} + \|\lambda_f - \lambda_g \|_{\infty} \right\}\,, \quad \left(f \,, g\right) \in \mathbb{D}_T^2\,.
$$
Some alternative characterizations of $M_1$-convergence can be found in \citet[Theorem~12.5.1]{whitt2002stochastic}.
It is important to note that the Skorokhod topologies $J_1$ and $M_1$ both imply local uniform convergence around continuity points of the limiting function. They only differ in the way they behave around discontinuity points of the limit.
\begin{proposition}[\citet{whitt2002stochastic}, Theorem~12.5.1]
    \label{proposition:M_1_local_uniform}
    If $d_{M_1}(x_n , x) \overset{n \to \infty}{\longrightarrow} 0\,$, then $x_n(T) \overset{n \to \infty}{\longrightarrow} x(T)$ and for each $t \notin Disc(x)\,$,
    $$
    \lim_{\delta \to 0^+}\, \limsup_{n \to \infty} \, v(x_n\,,x\,,t\,,\delta) = 0\,,
    $$
    where
    $$
    v(x_1\,,x_2\,,t\,,\delta) := \sup_{0 \, \vee \, (t - \delta) \leq t_1,t_2 \leq (t+\delta) \, \wedge \, T} \, |x_1(t_1) - x_2(t_2)|\,, \quad \delta > 0 \,, \quad t \leq T\,, \quad \left(x_1\,,x_2\right)\in \mathbb{D}_T^2\,,
    $$
    and $Disc(x) := \left\{t \in [0\,,T]\,:\, x(t-) \neq x(t)\right\}$ is the discontinuity set of $x\,$.
\end{proposition}
\noindent
Since the limit is càdlàg, its set of discontinuities is at most countable, and therefore $M_1$-convergence implies pointwise convergence outside a subset of null Lebesgue measure.

Relatively compact sets are also explicitly characterized for the $M_1$ topology.
\begin{proposition}[\citet{whitt2002stochastic}, Theorem~12.12.2]
    \label{proposition:M_1_compactness}
    A subset $A$ of $\mathbb{D}_T$ has compact closure for the $M_1$ topology if and only if the following two conditions hold:
    \begin{itemize}
        \item[$\bullet$] $\sup\limits_{x \in A} \, \|x\|_{\infty} < + \infty\,,$
        \item[$\bullet$] $\lim\limits_{\delta \to 0^+} \, \sup\limits_{x \in A}\, w(x,\delta) = 0\,$, where
        $$
w(x,\delta) := \sup_{0 \leq t \leq T}\, w'(x, t, \delta) \, \vee \, v(x,0,\delta) \,\vee \, v(x, T, \delta)\,, \quad \delta > 0\,, \quad x \in \mathbb{D}_T\,,
        $$
        with
        $$
        w'(x, t, \delta) := \sup_{0 \, \vee \, (t- \delta) \leq t_1 < t_2 < t_3 \leq (t+ \delta)\, \wedge \, T}\, d\left(x(t_2)\,, [x(t_1)\,,x(t_3)]\right)\,, \quad \delta > 0\,, \quad t \leq T\,, \quad x \in \mathbb{D}_T\,,
        $$
        and
        $$
        v(x,t,\delta) := \sup_{0 \, \vee \, (t-\delta)\leq t_1 \leq t_2 \leq (t+ \delta) \, \wedge \, T}\, |x(t_1) - x(t_2)|\,, \quad \delta > 0\,, \quad t \leq T \,, \quad x \in \mathbb{D}_T\,.
        $$
        \end{itemize}
    For $x \in \R$ and $I \subset \R\,$, $d(x, I)$ denotes the distance between $x$ and $I\,$,
    \begin{align}\label{eq:dM1}
    d(x, I) := \inf_{y \in I}\, |x - y|\,.    
    \end{align}
    \end{proposition}
\begin{remark}
    \label{remark:M_1_cv_implies_bounded}
In particular, the previous proposition shows that an $M_1$-convergent sequence has uniformly bounded uniform norm.
\end{remark}
\begin{remark}
    \label{remark:example_not_cv_M_1}
    If we define the sequence $\left(f_n\right)_{n \geq 2} := \left(\1_{[1/2 - 1/n\,, 1/2 + 1/n)}\right)_{n \geq 2}$ of elements of $\mathbb{D}_1\,$, we remark that for any $\delta > 0\,$, there exists $N \geq 2$ such that for any $n \geq N\,$, $w'\left(f_n\,, 1/2\,, \delta\right) = 1\,$. Therefore, $\left\{f_n\,,\, n \geq 2\right\}$ is not relatively compact in $\left(\mathbb{D}_1\,, M_1\right)$ and the sequence cannot converge for this topology.
\end{remark}
\begin{remark}
    \label{remark:addition_not_continuous_M_1}
    In Remark \ref{remark:example_not_cv_M_1}, 
    $$
f_n = 1 - \1_{[0\,,1/2 - 1/n)} - \1_{[1/2 + 1/n\,, 1]} \,, \quad n \geq 2\,.
    $$
    However, $\left(\1_{[0\,,1/2 - 1/n)}\right)_{n \geq 2}$ converges to $\1_{[0\,, 1/2)}$ and $\left(\1_{[1/2 + 1/n\,, 1]}\right)_{n\geq 2}$ to $\1_{[1/2\,,1]}$ in the $M_1$ topology (see \citet[Corollary~12.5.1]{whitt2002stochastic} for $M_1$-convergence of sequences of monotone functions), although we have shown that $\left(f_n\right)_{n \geq 2}$ does not converge to $0\,$. This example shows that addition is not sequentially continuous for the $M_1$ topology.
\end{remark}

We can now state the necessary and sufficient conditions for $M_1$-tightness of a sequence of stochastic processes in $\mathbb{D}_T\,$, which is linked to the characterization of compacts given by Proposition \ref{proposition:M_1_compactness}.
\begin{proposition}[\citet{whitt2002stochastic}, Theorem~12.12.3]
    \label{proposition:M_1_tightness}
    A sequence $(X^n)_{n \geq 0}$ of processes in $\mathbb{D}_T$ endowed with the $M_1$ topology is tight if and only if the following two conditions hold:
    \begin{itemize}
        \item[$\bullet$] $\lim\limits_{R \to + \infty}\, \sup\limits_n \, \mathbb{P}\left(\|X^n\|_{\infty} > R\right) = 0\,,$ 
        \item[$\bullet$] for any $\eta > 0\,$, $\lim\limits_{\delta \to 0^+}\, \sup\limits_n \,\mathbb{P}\left(w(X^n\,,\delta) \geq \eta  \right) = 0\,,$
    \end{itemize}
    where the modulus $w$ is defined in Proposition \ref{proposition:M_1_compactness}.
\end{proposition}
\begin{remark}
    \label{remark:M1_non_decreasing_processes} For a sequence $(X^n)_{n \geq 0}$ of nondecreasing processes, the modulus $w$ simplifies. We remark that $w'\left(X^n, t, \delta \right) = 0$, along with $v\left(X^n, 0, \delta\right) = X^n_{\delta} - X^n_0$ and $v\left(X^n, T, \delta \right) = X^n_T - X^n_{T - \delta}\,$.
\end{remark}

We conclude this section with a property of the Borel $\sigma$-field generated by the $M_1$ topology. We first need to define the usual $\sigma$-field on $\mathbb{D}_T\,$.
\begin{definition}
\label{def:sigma_field}
    Define the real-valued functionals on $\mathbb{D}_T\,$, called evaluation functionals, as
    $$
    \pi_t : x \longmapsto x(t) \,, \quad t \leq T\,.
    $$
    Then, we denote by $\mathcal{F}_{\mathbb{D}_T}$ the $\sigma$-field they generate:
    $$
    \mathcal{F}_{\mathbb{D}_T} := \sigma \left(\pi_t\,,\,\, t \leq T\right)\,.
    $$
\end{definition}
\begin{remark}
\label{remark:sigma_field}
    $\mathcal{F}_{\mathbb{D}_T}$ is the usual $\sigma$-field for càdlàg stochastic processes. For a family of real-valued functions $X := \left(X_t\right)_{t \leq T}$ defined on a measurable space $\left( \Omega\,, \mathcal{F}\right)\,$, one can check that the following properties are equivalent:
    \begin{itemize}
        \item[$\bullet$] $X_t : \Omega \to \R$ is $\left(\mathcal{F}\,, \mathcal{B}_{\R}\right)$-measurable, for all $0\leq t \leq T\,$.
        \item[$\bullet$] $X : \Omega \to \mathbb{D}_T$ is $\left( \mathcal{F}\,, \mathcal{F}_{\mathbb{D}_T}\right)$-measurable.
    \end{itemize}
\end{remark}
\noindent
The announced result on the $M_1$ topology is as follows.
\begin{proposition}[\citet{whitt2002stochastic}, Theorem~11.5.2]
    \label{proposition:M_1_borel_field}
    The Borel $\sigma$-field generated by the $M_1$ topology coincides with $\mathcal{F}_{\mathbb{D}_T}$ generated by the evaluations.
\end{proposition}

\section{The $S$ topology}\label{app:S_topology}
\citet{jakubowski1997non} introduced a non-Skorokhod topology on $\mathbb{D}_T$ named the $S$ topology.
We first define the convergence $\longrightarrow_S\,$. We say that $x_n \longrightarrow_S x$ in $\mathbb{D}_T$ if for any $\varepsilon > 0\,$, there exist functions $\left(v_{n, \varepsilon}\right)_n$ and $v_{\varepsilon}$ of bounded variation such that:
\begin{itemize}
    \item[$\bullet$] $\|x_n - v_{n, \varepsilon}\|_{\infty} \leq \varepsilon$ for any $n\,$, and $\|x - v_{\varepsilon}\|_{\infty} \leq \varepsilon\,$.
    \item[$\bullet$] $(dv_{n, \varepsilon})_n$ converges weakly to $dv_{\varepsilon}\,$, that is, for any continuous function $f$ on $[0\,, T]\,$,
    $$
    \int_0^T \, f(t)\, dv_{n, \varepsilon}(t) \overset{n \to \infty}{\longrightarrow} \int_0^T \, f(t) \, dv_{\varepsilon}(t)\,.
    $$
\end{itemize}
From this sequential convergence, one constructs a topology named $S$ on $\mathbb{D}_T\,$. This is done following \cite{kisynski1960convergence}: a subset $A \subset \mathbb D_T$ is open if for any sequence $\longrightarrow_S$-converging to an element of $A\,$, all elements of the sequence, except a finite number, belong to $A\,$. Note that the convergence in the $S$ topology, denoted by $\overset{S}{\longrightarrow}\,$, is generally weaker than the initial convergence $\longrightarrow_S\,$. In fact, we have:

\begin{center}
    \textit{$x_n \overset{S}{\longrightarrow} x$ if and only if, in any subsequence $\left(x_{n_k}\right)_k\,$, there exists a further subsequence $\left(x_{n_{k_l}}\right)_l$ such that $x_{n_{k_l}} \longrightarrow_S x\,$.}
\end{center}
In particular, continuity of real-valued functions for the $S$ topology coincides with sequential continuity for $\longrightarrow_S\,$. Additionally, this gives the following link to pointwise convergence.
\begin{proposition}[\citet{jakubowski1997non}, Corollary~2.9]
    \label{propostion:S_pointwise_cv}
    If $x_n \overset{S}{\longrightarrow} x\,$, then for each subsequence $\left(x_{n_k}\right)_k\,$, there exist a further subsequence $\left(x_{n_{k_l}}\right)_l$ and a countable set $Q \subset [0\,,T)$ such that 
    $$
    x_{n_{k_l}}(t) \longrightarrow x(t)\,, \quad t \in [0\,,T] \setminus Q\,.
    $$
\end{proposition}
We next list some useful properties of the $S$ topology.
\begin{proposition}[\citet{jakubowski1997non}, Theorem~2.13, Lemma~2.8, Corollary~2.11]
    \label{proposition:S_properties}
\hspace{1mm}
\begin{itemize}
    \item[$\bullet$] The $S$ topology is weaker than the $M_1$ topology.
    \item[$\bullet$] $(\mathbb{D}_T\,, S)$ is not metrizable, and therefore not Polish.
    \item[$\bullet$] There exists a countable family of $S$-continuous functionals $(f_n)_{n\geq 0}$ separating points in $\mathbb{D}_T\,$, that is, for any $x\,, y$ in $\mathbb{D}_T\,$:
    $$
    f_n(x) = f_n(y)\,, \,  \, n \geq 0 \implies x = y\,.
    $$
    \item[$\bullet$] The Borel $\sigma$-algebra generated by the $S$ topology coincides with $\mathcal{F}_{\mathbb{D}_T}$ generated by the evaluations (see Definition \ref{def:sigma_field}).
    \item[$\bullet$] Addition is sequentially continuous for the $S$ topology. 
    \item[$\bullet$] If $x_n \overset{S}{\longrightarrow} x\,$, then $\sup\limits_n \, \|x_n\|_{\infty} < + \infty\,.$ 
    \item[$\bullet$] Let $\Phi : [0\,,T] \times \R \to \R$ measurable such that for all $t\,$, $\Phi(t\,, \cdot)$ is continuous. If for all $C > 0\,$, 
    $$\sup_{0\leq t \leq T}\, \sup_{|x| \leq C} \, |\Phi(t,x)| < +\infty\,,$$
    then the functional $f : (\mathbb{D}_T\,, S) \to \R$ defined by 
    $$
    f(x) = \int_0^T\, \Phi(t\,,x(t))\,dt \,, \quad x \in \mathbb{D}_T\,,
    $$
    is continuous.
\end{itemize}
\end{proposition}
\begin{remark}
    \label{remark:S_cv_pointwise_at_T}
    If $x_n \overset{S}{\longrightarrow} x\,$, Proposition \ref{proposition:S_properties} shows in particular that $\sup_n \,|x_n(T)| < + \infty\,$. Moreover, for any convergent subsequence $(x_{n_k}(T))_k\,$, we can extract a further subsequence $(x_{n_{k_l}}(T))_l$ such that $x_{n_{k_l}}(T) \to x(T)$ from Proposition \ref{propostion:S_pointwise_cv}. Hence, we deduce that $x_n(T) \to x(T)$ without extraction.
\end{remark}
In order to compare the $M_1$ and the $S$ topology, one can keep in mind the sequence of functions $\left(f_n\right)_{n \geq 2} := \left(\1_{[1/2 - 1/n\,, 1/2 + 1/n)}\right)_{n \geq 2}$ from Remarks \ref{remark:example_not_cv_M_1} and \ref{remark:addition_not_continuous_M_1} that does not converge in the $M_1$ topology. For any $n \geq 2\,$, the function $f_n$ has bounded variation and it is clear that we have the weak convergence of measures:
$$
df_n = \delta_{1/2 - 1/n} - \delta_{1/2 + 1/n} \overset{n \to \infty}{\longrightarrow} 0\,,
$$
so that $f_n \overset{S}{\longrightarrow} 0\,$. This limit can also be understood in light of Remark \ref{remark:addition_not_continuous_M_1}, since $M_1$ is coarser than $S$ and addition is sequentially continuous for the latter topology.

Regarding the tightness of a sequence of càdlàg processes $(X^n)_{n\geq 0}\,$, we have the following simple criterion.
\begin{proposition}[\citet{jakubowski1997non}, Proposition~3.1]
    \label{proposition:S_tightness}
    A sequence of càdlàg processes $(X^n)_{n\geq 0}$ is $S$-tight if and only if the following two conditions hold:
    \begin{itemize}
        \item[$\bullet$] The sequence $(\left\|X^n\right\|_{\infty})_{n \geq 0}$ is tight.
        \item[$\bullet$] For any $a < b\,$, the sequence $\left(N^{a,b}(X^n)\right)_{n \geq 0}$ is tight, where $N^{a,b}(x)$ is the number of up-crossings of $x$ at levels $a<b\,$. More precisely, $N^{a,b}(x)$ is the greatest integer $k \geq 0$ such that there exist times $0\leq t_1 < t_2 < \ldots < t_{2k} \leq T$ satisfying $x(t_{2i - 1}) < a$ and $x(t_{2i}) > b$ for $i = 1\,,\ldots\,,k\,$.
    \end{itemize}
\end{proposition}
\noindent
Thanks to Doob's inequalities, these conditions are easily satisfied for a sequence of supermartingales (see the discussion at the beginning of \citet[Section~4]{jakubowski1997non}).
\begin{proposition}
    \label{proposition:S_tightness_supermartingale}
    Let $(M^n)_{n \geq 0}$ be a sequence of càdlàg supermartingales. If 
    $$
    \sup_{n}\, \sup_{0 \leq t \leq T}\, \mathbb{E}\left[\left|M^n_t \right|\right] < + \infty\,,
    $$
    then the sequence is $S$-tight.
\end{proposition}
Furthermore, under certain conditions, one has the stability of the martingale property along $S$-convergence. We first need to define a form of \textit{almost sure compact-valued representation} for convergence in the $S$ topology, which is stronger than the convergence in distribution. In fact, it is the only sequential topology on the space of probability measures on $\mathbb{D}_T\,$, which is finer than the weak topology, and for which relative compactness coincides with $S$-tightness (see \citet{jakubowski2000convergence} for more details).
\begin{definition}[\citet{jakubowski1997non}, Definition~3.3]
    \label{definition:cv_D}
    For a sequence of càdlàg processes $\left(X^n\right)_{n \geq 0}\,$, we say that $X^n \overset{*}{\longrightarrow}_{\mathcal{D}} X$ if for every subsequence $\left(X^{n_k}\right)_{k \geq 0}$ there exists a further subsequence $\left(X^{n_{k_l}}\right)_{l \geq 0}$ and stochastic processes $\left(\hat{X}^l\right)_{l \geq 0}$ and $\hat{X}$ defined on the probability space $\left([0\,,1]\,, \mathcal{B}_{[0\,,1]}\,, dt\right)$ such that:
    \begin{itemize}
        \item[$\bullet$] For any $l \geq 0\,$, $X^{n_{k_l}} \sim \hat{X}^l$\,, and $X \sim \hat{X}\,$.
        \item[$\bullet$] For any $\omega \in [0\,,1]\,$, $\hat{X}^l(\omega) \overset{S}{\longrightarrow} \hat{X}(\omega)\,$.
        \item[$\bullet$] For any $\varepsilon > 0\,$, there exists an $S$-compact subset $K_{\varepsilon} \subset \mathbb{D}_T$ such that
        $$
        \mathbb{P}\left(\left\{\hat{X}^l \in K_{\varepsilon} \,, \,  l \geq 0\right\}\right) > 1 - \varepsilon\,.
        $$
    \end{itemize}
\end{definition}
\noindent
We can now state and prove the stability of the martingale property. The proof is inspired by the one of \citet*[Lemma~3.7]{guo2017tightness}.
\begin{proposition}
    \label{proposition:S_martingale_stability}
    Let $(M^n)_{n \geq 0}$ be a sequence of martingales on $[0\,, T]$ such that $M^n \overset{*}{\longrightarrow}_{\mathcal{D}} M\,$. If $\sup\limits_n \, \mathbb{E}\left[\sup\limits_{0 \leq t \leq T}\, \left|M^n_t\right|^2 \right] < + \infty\,$, then $M$ is a square-integrable martingale on $[0\,, T]$ with respect to its own filtration.
\end{proposition}
\begin{proof}
    From \citet[Theorem~3.11]{jakubowski1997non}, $\overset{*}{\longrightarrow}_{\mathcal{D}}$-convergence implies the existence of a subsequence $\left(M^{n_k}\right)_{k \geq 0}$ and of a countable set $Q \subset [0\,,T)\,$, such that for any $r \in \N$ and any $t_0 < t_1 < \ldots < t_r$ in $[0\,,T] \setminus Q\,$, we have the finite-dimensional weak convergence
    $$
\left(M^{n_k}_{t_0}\,, M^{n_k}_{t_1}\,, \ldots \,, M^{n_k}_{t_r}\right) \implies \left(M_{t_0}\,, M_{t_1}\,, \ldots \,, M_{t_r} \right) \,.
    $$
    Furthermore, for any $t \leq T\,$, the sequence $\left(M^{n_k}_t\right)_{k \geq 0}$ is uniformly integrable since 
    $$
    C := \sup_n \, \mathbb{E}\left[\sup_{0\leq t \leq T}\left|M^n_t \right|^2\right] < \infty\,.
    $$
    This first implies that $M_t$ is square-integrable for any $t \in [0\,,T] \setminus Q$ with 
    $$
    \mathbb{E}\left[\left|M_t\right|^2 \right]\leq C\,, \quad t \in [0\,,T] \setminus Q\,.
    $$
    Therefore, for any sequence $\left(t_n\right)_{n \geq 0}$ in $[0\,,T] \setminus Q\,$, $\left(M_{t_n}\right)_{n \geq 0}$ is uniformly integrable, so that by right-continuity of $M\,$,
    \begin{equation}
        \label{eq:M_square_integrable}
\mathbb{E}\left[\left|M_t \right|^2\right] \leq C < + \infty \,, \quad t \leq T\,.
    \end{equation}
    Then, let $s < t$ in $[0\,,T] \setminus Q\,$, $r \in \N$ and $t_0 < t_1 <\ldots <t_r$ in $[0\,, s] \setminus Q\,$, as well as $f_0\,,f_1\,,\ldots\,,f_r$ bounded real-valued continuous functions on $\R\,$. We have 
    $$
    \mathbb{E}\left[f_0\left(M^{n_k}_{t_0}\right)\,f_1\left(M^{n_k}_{t_1}\right)\,\ldots\,f_r\left(M^{n_k}_{t_r}\right)\, \left(M^{n_k}_t - M^{n_k}_s\right)\right] = 0\,, \quad k \geq 0\,,
    $$
    by martingality. From the finite-dimensional convergence, as well as uniform integrability of $\left(M^{n_k}_t\right)_{k \geq 0}$ and $\left(M^{n_k}_s\right)_{k \geq 0}\,$, we obtain, in the limit $k \to \infty\,$,
    $$
    \mathbb{E}\left[f_0\left(M_{t_0}\right)\,f_1\left(M_{t_1}\right)\,\ldots\,f_r\left(M_{t_r}\right)\, \left(M_t - M_s\right)\right] = 0\,.
    $$
    This being true for $t_0\,, t_1\,, \ldots \,, t_r \,, t\,, s$ in $[0\,,T] \setminus Q\,$, by the same argument we used to obtain \eqref{eq:M_square_integrable}, we get for any $r \in \N\,$ and bounded continuous $f_0\,, f_1\,, \ldots \,, f_r\,$,
    $$
\mathbb{E}\left[f_0\left(M_{t_0}\right)\,f_1\left(M_{t_1}\right)\,\ldots\,f_r\left(M_{t_r}\right)\, \left(M_t - M_s\right)\right] = 0 \,, \quad t_0 < t_1 < \ldots < t_r \leq s < t \leq T\,.
    $$
    Finally, for any Borel set $B \in \mathcal{B}_{\R}\,$, one can find a sequence of bounded continuous functions converging pointwise to $\1_{B}$ while staying uniformly bounded. Therefore, we get for any $r \in \N\,$, and any $t_0 < \ldots < t_r \leq s < t \leq T\,$,
    $$
\mathbb{E}\left[\1_{M_{t_0} \in B_0\,, M_{t_1} \in B_1\,, \ldots \,, M_{t_r} \in B_r}\, \left(M_t - M_s\right)\right] = 0\,, \quad B_0\,, \ldots \,, B_r \in \mathcal{B}_{\R}\,,
$$
so that by the monotone class theorem
$$
\mathbb{E}\left[M_t \, | \, \mathcal{F}_s\right] = M_s \,, \quad s < t \leq T\,,
$$
where $\left(\mathcal{F}_t\right)_{t \leq T} := \left(\sigma\left(M_s \,, \, 0 \leq s \leq t\right)\right)_{t \leq T}$ is the filtration generated by the process $M\,$. This, along with \eqref{eq:M_square_integrable}, concludes the proof.
\end{proof}

The $S$ topology is different from the Skorokhod topologies in that it is not metrizable. In particular, the usual Skorokhod representation theorem \citep[Theorem~6.7]{billingsley2013convergence} cannot be readily applied. However, Jakubowski developed a version of it in the case of general topological spaces admitting a countable family of continuous functionals separating their points (see Proposition \ref{proposition:S_properties} for a definition).
\begin{proposition}[\citet{jakubowski1998almost}, Theorem~2, Theorem~3]
    \label{proposition:skorokhod_representation_non_metrizable}
    Let $(E, \tau)$ be a topological space such that there exists a countable family of $\tau$-continuous functionals separating points of $E\,$, in the sense of Proposition \ref{proposition:S_properties}. Let $(X^n)_{n \geq 0}$ be a $\tau$-tight sequence of $E$-valued random elements. Then, one can find a subsequence $\left(X^{n_k}\right)_{k \geq 0}$ as well as $E$-valued random elements $\left(\hat{X}^k\right)_{k \geq 0}$ and $\hat{X}$ defined on the common probability space $\left([0\,, 1]\,, \mathcal{B}_{[0\,, 1]}\,, dt\right)$ such that:
    \begin{itemize}
        \item[$\bullet$] For any $k\geq 0\,$, $X^{n_k} \sim \hat{X}^k\,.$
        \item[$\bullet$] For any $\omega \in [0\,,1]\,$, $\hat{X}^k(\omega) \overset{\tau}{\to} \hat{X}(\omega)$\,.
        \item[$\bullet$] For any $\varepsilon >0\,$, there exists a $\tau$-compact subset $K_{\varepsilon} \subset E$ such that
        $$
        \mathbb{P}\left(\left\{\hat{X}^k \in K_{\varepsilon}\,, \, k \geq 0 \right\}\right) > 1 - \varepsilon\,.
        $$
    \end{itemize}
\end{proposition}
\begin{remark}
\label{remark:product_sigma_field}
    In Proposition \ref{proposition:skorokhod_representation_non_metrizable}, an $E$-valued random element is defined as a measurable function $X : \left(\Omega\,, \mathcal{F}\right) \to \left(E\,, \mathcal{B}_{\tau}\right)$ where $\left(\Omega\,, \mathcal{F}\right)$ is a measurable space and $\mathcal{B}_{\tau}$ is the Borel $\sigma$-field generated by the topology $\tau\,$. In the case of a couple of $\mathbb{D}_T$-valued random variables $\left(X\,,M\right)\,$, both are measurable with respect to $\mathcal{F}_{\mathbb{D}_T}$ generated by the evaluations (see Definition \ref{def:sigma_field} and Remark \ref{remark:sigma_field}) and the couple is measurable with respect to the product $\sigma$-field $\mathcal{F}_{\mathbb{D}_T} \otimes \mathcal{F}_{\mathbb{D}_T}\,$. Since in the present work we consider the product topology $M_1 \otimes S$ on $\mathbb{D}_T^2$, we have to check that $\left(X\,,M\right)$ is measurable with respect to its Borel $\sigma$-field $\mathcal{B}_{M_1 \otimes S}\,$. From Propositions \ref{proposition:M_1_borel_field} and \ref{proposition:S_properties},
    $$
    \mathcal{B}_{M_1} = \mathcal{B}_S = \mathcal{F}_{\mathbb{D}_T}\,,
    $$
    so that
    $$
    \mathcal{F}_{\mathbb{D}_T} \otimes \mathcal{F}_{\mathbb{D}_T} = \mathcal{B}_{M_1} \otimes \mathcal{B}_S \subset \mathcal{B}_{M_1 \otimes S}\,,
    $$
    where the inclusion is a classical result of measure theory. Finally, the product topology $M_1 \otimes S$ is included in $M_1 \otimes M_1$ which has a countable basis of open sets (since it is Polish) so that
    $$
    \mathcal{B}_{M_1 \otimes M_1} = \mathcal{B}_{M_1} \otimes \mathcal{B}_{M_1} = \mathcal{F}_{\mathbb{D}_T} \otimes \mathcal{F}_{\mathbb{D}_T}\,,
    $$
    and thus
    $$
    \mathcal B_{M_1 \otimes S} \subseteq \mathcal B_{M_1 \otimes M_1} = \mathcal F_{\mathbb D_T} \otimes \mathcal F_{\mathbb D_T}\,.
    $$
    Finally, we have proved
    $$
    \mathcal B_{M_1 \otimes S} = \mathcal F_{\mathbb D_T} \otimes \mathcal F_{\mathbb D_T}\,.
    $$
    We are therefore allowed to apply Proposition \ref{proposition:skorokhod_representation_non_metrizable} in our context. Moreover, the obtained processes defined on $\left( [0,1]\,, \mathcal{B}_{[0\,,1]}\right)$ are measurable with respect to $\mathcal{F}_{\mathbb{D}_T} \otimes \mathcal{F}_{\mathbb{D}_T}\,$.
\end{remark}
\end{appendix}

\bibliographystyle{plainnat}
\bibliography{ref}

@article{abi2021weak,
  title={Weak existence and uniqueness for affine stochastic {V}olterra equations with {$L^1$}-kernels},
  author={Abi Jaber, Eduardo},
  journal={Bernoulli},
  volume = {27},
  number = {3},
  pages = {1583--1615},
  year={2021}
}

@article{abijaber2025simulating,
  title={Simulating integrated {V}olterra square-root processes and {V}olterra {H}eston models via {I}nverse {G}aussian},
  author={Abi Jaber, Eduardo and Attal, Elie},
  journal={arXiv preprint arXiv:2504.19885},
  year={2025}
}

@article{abi2024reconciling,
  title={Reconciling rough volatility with jumps},
  author={Abi Jaber, Eduardo and De Carvalho, Nathan},
  journal={SIAM Journal on Financial Mathematics},
  volume={15},
  number={3},
  pages={785--823},
  year={2024}
}

@book{gripenberg1990volterra,
  title={Volterra integral and functional equations},
  author={Gripenberg, Gustaf and Londen, Stig-Olof and Staffans, Olof},
  year={1990},
  publisher={Cambridge University Press}
}

@book{whitt2002stochastic,
  author={Whitt, Ward},
  title={Stochastic-process limits: an introduction to stochastic-process limits and their application to queues},
  series={Springer Series in Operations Research},
  publisher={Springer},
  year={2002}
}

@article{sojmark2023weak,
  title={Weak Convergence of Stochastic Integrals on {S}korokhod Space in {S}korokhod's {J}1 and {M}1 Topologies},
  author={S{\o}jmark, Andreas and Wunderlich, Fabrice},
  journal={arXiv preprint arXiv:2309.12197},
  year={2023}
}

@article{jakubowski1997non,
  title={A non-{S}korohod topology on the {S}korohod space},
  author={Jakubowski, Adam},
  year={1997},
  journal={Electronic Journal of Probability},
  volume={2},
  number={4},
  pages={1--21}
}

@article{jakubowski1998almost,
  title={The almost sure {S}korokhod representation for subsequences in nonmetric spaces},
  author={Jakubowski, Adam},
  journal={Theory of Probability \& Its Applications},
  volume={42},
  number={1},
  pages={167--174},
  year={1997},
  publisher={SIAM}
}

@article{jusselin2020no,
  title={No-arbitrage implies power-law market impact and rough volatility},
  author={Jusselin, Paul and Rosenbaum, Mathieu},
  journal={Mathematical Finance},
  volume={30},
  number={4},
  pages={1309--1336},
  year={2020},
  publisher={Wiley Online Library}
}

@book{protter2005stochastic,
  title={Stochastic Integration and Differential Equations},
  author={Protter, Philip},
  year={1990},
  series={Applications of Mathematics},
  number={21},
  publisher={Springer}
}

@article{abi2019affine,
  title={Affine {V}olterra processes},
  author={Abi Jaber, Eduardo and Larsson, Martin and Pulido, Sergio},
  journal = {Annals of Applied Probability},
  volume = {29},
  number = {5},
  pages = {3155--3200},
  year={2019}
}

@article{skorokhod1956limit,
  title={Limit theorems for stochastic processes},
  author={Skorokhod, Anatoliy V.},
  journal={Theory of Probability \& Its Applications},
  volume={1},
  number={3},
  pages={261--290},
  year={1956},
  publisher={SIAM}
}

@book{billingsley2013convergence,
  title={Convergence of probability measures},
  author={Billingsley, Patrick},
  year={1999},
  edition={2nd},
  publisher={John Wiley \& Sons}
}

@article{abi2019multifactor,
  title={Multifactor approximation of rough volatility models},
  author={Abi Jaber, Eduardo and El Euch, Omar},
  journal={SIAM Journal on Financial Mathematics},
  volume={10},
  number={2},
  pages={309--349},
  year={2019},
  publisher={SIAM}
}

@article{jaisson2016rough,
  title={Rough fractional diffusions as scaling limits of nearly unstable heavy tailed {H}awkes processes},
  author={Jaisson, Thibault and Rosenbaum, Mathieu},
  journal = {Annals of Applied Probability},
  volume = {26},
  number = {5},
  pages = {2860--2882},
  year={2016}
}

@article{el2019characteristic,
  title={The characteristic function of rough {H}eston models},
  author={El Euch, Omar and Rosenbaum, Mathieu},
  journal={Mathematical Finance},
  volume={29},
  number={1},
  pages={3--38},
  year={2019},
  publisher={Wiley Online Library}
}

@article{gatheral2018volatility,
  title={Volatility is rough},
  author={Gatheral, Jim and Jaisson, Thibault and Rosenbaum, Mathieu},
  journal={Quantitative Finance},
  volume={18},
  number={6},
  pages={933--949},
  year={2018},
  publisher={Taylor \& Francis}
}

@book{tankov2003financial,
  title={Financial modelling with jump processes},
  author={Cont, Rama and Tankov, Peter},
  year={2004},
  series={Financial Mathematics Series},
  publisher={Chapman and Hall/CRC}
}

@article{kern2022skorokhod,
  title={Skorokhod topologies: {W}hat they are and why we should care},
  author={Kern, Julian},
  journal={Mathematische Semesterberichte},
  volume={71},
  pages={1--18},
  year={2024}
}

@article{mandelbrot1968fractional,
  title={Fractional {B}rownian motions, fractional noises and applications},
  author={Mandelbrot, Benoit B. and Van Ness, John W.},
  journal={SIAM review},
  volume={10},
  number={4},
  pages={422--437},
  year={1968},
  publisher={SIAM}
}

@article{levy1953random,
  title={Random functions: general theory with special reference to {L}aplacian random functions},
  author={L{\'e}vy, Paul},
  volume={1},
  pages={331--390},
  year={1953},
  journal={University of California Publications in Statistics}
}

@misc{barndorff2012basics,
  title={Basics of {L}{\'e}vy processes},
  author={Barndorff-Nielsen, Ole E and Shephard, Neil},
  year={2012},
  howpublished = {draft chapter from a book by the authors on ``{L}{\'e}vy Driven Volatility Models''},
  url = {https://public.econ.duke.edu/~get/browse/courses/883/Spr15/COURSE-MATERIALS/Z_Papers/BNS2012.pdf}
}

@book{applebaum2009levy,
  title={L{\'e}vy Processes and Stochastic Calculus},
  author={Applebaum, David},
  year={2009},
  edition={2nd},
  publisher={Cambridge University Press}
}

@book{gorenflo2020mittag,
  title={Mittag-Leffler functions, related topics and applications},
  author={Gorenflo, Rudolf and Kilbas, Anatoly A. and Mainardi, Francesco and Rogosin, Sergei},
  year={2020},
  edition={2nd},
  publisher={Springer}
}

@article{meyer1984tightness,
  title={Tightness criteria for laws of semimartingales},
  author={Meyer, Paul Andr{\'e} and Zheng, Weian},
  journal={Annales de l'I.H.P., Section B},
  volume={20},
  number={4},
  pages={353--372},
  year={1984}
}

@article{guo2017tightness,
  title={Tightness and duality of martingale transport on the {S}korokhod space},
  author={Guo, Gaoyue and Tan, Xiaolu and Touzi, Nizar},
  journal={Stochastic Processes and their Applications},
  volume={127},
  number={3},
  pages={927--956},
  year={2017},
  publisher={Elsevier}
}

@article{jakubowski2000convergence,
  title={From convergence of functions to convergence of stochastic processes. {O}n {S}korokhod’s sequential approach to convergence in distribution},
  author={Jakubowski, Adam},
  note={In \textit{Skorokhod's Ideas in Probability Theory}, V. Korolyuk, N. Portenko, and H. Syta, pages 179-194. Institute of Mathematics, National Academy of Sciences of Ukraine, Kyiv},
  year = {2000}
}

@article{dawson1994super,
  title={A super-{B}rownian motion with a single point catalyst},
  author={Dawson, Donald A. and Fleischmann, Klaus},
  journal={Stochastic Processes and their Applications},
  volume={49},
  number={1},
  pages={3--40},
  year={1994},
  publisher={Elsevier}
}

@article{mytnik2015uniqueness,
  title={Uniqueness for {V}olterra-type stochastic integral equations},
  author={Mytnik, Leonid and Salisbury, Thomas S.},
  journal={arXiv preprint arXiv:1502.05513},
  year={2015}
}

@article{fleischmann1995new,
  title={A new approach to the single point catalytic super-{B}rownian motion},
  author={Fleischmann, Klaus and Le Gall, Jean-Fran{\c{c}}ois},
  journal={Probability Theory and Related Fields},
  volume={102},
  pages={63--82},
  year={1995},
  publisher={Springer}
}

@book{jacod2013limit,
  title={Limit theorems for stochastic processes},
  author={Jacod, Jean and Shiryaev, Albert N.},
  series={Grundlehren der
mathematischen Wissenschaften},
  volume={288},
  year={2013},
  edition={2nd},
  publisher={Springer}
}

@article{euch2018perfect,
  title={Perfect hedging in rough {H}eston models},
  author={El Euch, Omar and Rosenbaum, Mathieu},
  journal={The Annals of Applied Probability},
  volume={28},
  number={6},
  pages={3813--3856},
  year={2018},
  publisher={JSTOR}
}

@inproceedings{kisynski1960convergence,
  title={Convergence du type {L}},
  author={Kisy{\'n}ski, J},
  booktitle={Colloquium Mathematicum},
  volume={7},
  number={2},
  pages={205--211},
  year={1960},
  organization={Polska Akademia Nauk. Instytut Matematyczny PAN}
}

@article{mechkov2015fast,
  title={Fast-reversion limit of the {H}eston model},
  author={Mechkov, Serguei},
  journal={Available at SSRN 2418631},
  year={2015}
}

@article{forde2011large,
  title={The large-maturity smile for the {H}eston model},
  author={Forde, Martin and Jacquier, Antoine},
  journal={Finance and Stochastics},
  volume={15},
  number={4},
  pages={755--780},
  year={2011},
  publisher={Springer}
}

@phdthesis{McCrickerd2021,
  author       = {Ryan McCrickerd},
  title        = {On Spatially Irregular Ordinary Differential Equations and a Pathwise Volatility Modelling Framework},
  school       = {Imperial College London},
  year         = {2021},
  type         = {Ph.D. thesis},
  doi          = {10.25560/92202},
  url          = {https://doi.org/10.25560/92202}
}

@article{bondi2025vy,
  title={L{é}vy processes as weak limits of rough {H}eston models},
  author={Bondi, Alessandro and Forde, Martin},
  journal={arXiv preprint arXiv:2508.14835},
  year={2025}
}

\end{document}